\documentclass[a4paper,oneside,10pt]{article}%
\usepackage{amsmath}
\usepackage{amsfonts}
\usepackage{amssymb}
\usepackage{graphicx}
\usepackage{color}
\usepackage[square,numbers,sort&compress]{natbib}%
\providecommand{\U}[1]{\protect \rule{.1in}{.1in}}

\newtheorem{theorem}{Theorem}[section]

\newtheorem{definition}[theorem]{Definition}

\newtheorem{example}[theorem]{Example}

\newtheorem{lemma}[theorem]{Lemma}

\newtheorem{proposition}[theorem]{Proposition}
\newtheorem{remark}[theorem]{Remark}

\newenvironment{proof}[1][Proof]{\noindent \textbf{#1.} }{\  $\Box$}
\numberwithin{equation}{section}

\begin{document}

\title{Maximum principle for stochastic optimal control problem under convex expectation}
\author{Xiaojuan Li\thanks{Department of Mathematics, Qilu Normal University, Jinan
250200, China. lxj110055@126.com. }
\and Mingshang Hu \thanks{Zhongtai Securities Institute for Financial Studies,
Shandong University, Jinan, Shandong 250100, China. humingshang@sdu.edu.cn.
Research supported by NSF (No. 12326603, 11671231).} }
\maketitle

\textbf{Abstract}. In this paper, we study a stochastic optimal control
problem under a type of consistent convex expectation dominated by
$G$-expectation. By the separation theorem for convex sets, we get the
representation theorems for this convex expectation and conditional convex
expectation. Based on these results, we obtain the variational equation for
cost functional by weak convergence and discretization methods. Furthermore,
we establish the maximum principle which is sufficient under usual convex
assumptions. Finally, we study the linear quadratic control problem by using
the obtained maximum principle.



\addcontentsline{toc}{section}{\hspace*{1.8em}Abstract}

\section{Introduction}

In economics and finance, there is a great deal of research on volatility
uncertainty (see \cite{DM, DK, EJ-2, FPW, MPZ, Peng2004} and the references
therein), which is characterized by a family of nondominated probability
measures. In particular, Peng \cite{Peng2004} proposed a consistent nonlinear
expectation method for the first time. On the basis of this work, Peng
\cite{P07a, P08a} established the theory of $G$-expectation $(\mathbb{\hat{E}%
}_{t})_{t\geq0}$, which contains the results of stochastic differential
equation driven by $G$-Brownian motion $B=(B^{1},\ldots,B^{d})^{T}$ ($G$-SDE)
and can be used to deal with problems involving volatility uncertainty. The
representation theorems for $G$-expectation $\mathbb{\hat{E}}$ and conditional
$G$-expectation $\mathbb{\hat{E}}_{t}$ were first obtained in \cite{DHP11} and
\cite{STZ} respectively. Hu et al. \cite{HJPS1, HJPS} established the theory
of backward stochastic differential equation driven by $G$-Brownian motion
($G$-BSDE). By the different formulation and method, Soner et al. \cite{STZ11}
established the theory of a new type of fully nonlinear BSDE, called 2BSDE.
For the latest developments in $G$-BSDE and 2BSDE, the reader may refer to
\cite{HJL1, HYH, HTW, LRT, PZ} and the references therein.

Under $G$-expectation $(\mathbb{\hat{E}}_{t})_{t\geq0}$ (sublinear
expectation), Hu and Ji \cite{HJ1} studied the following forward-backward
control system:
\begin{equation}
\left \{
\begin{array}
[c]{rl}%
dX_{t}^{u}= & b(t,X_{t}^{u},u_{t})dt+\sum_{i,j=1}^{d}h^{ij}(t,X_{t}^{u}%
,u_{t})d\langle B^{i},B^{j}\rangle_{t}+\sum_{i=1}^{d}\sigma^{i}(t,X_{t}%
^{u},u_{t})dB_{t}^{i},\\
X_{0}^{u}= & x_{0}\in \mathbb{R}^{n},
\end{array}
\right.  \label{new-e1}%
\end{equation}%
\begin{equation}
\left \{
\begin{array}
[c]{rl}%
dY_{t}^{u}= & -f(t,X_{t}^{u},Y_{t}^{u},Z_{t}^{u},u_{t})dt-\sum_{i,j=1}%
^{d}g^{ij}(t,X_{t}^{u},Y_{t}^{u},Z_{t}^{u},u_{t})d\langle B^{i},B^{j}%
\rangle_{t}\\
& +Z_{t}^{u}dB_{t}+dK_{t}^{u},\\
Y_{T}^{u}= & \Phi(X_{T}^{u}),
\end{array}
\right.  \label{new-e2}%
\end{equation}
where the control domain $U$ is a given nonempty convex set of $\mathbb{R}%
^{m}$, and a control $u$ is admissible means that the above forward-backward
SDE driven by $G$-Brownian motion has a unique solution with some
integrability. The cost functional is defined by $J(u):=Y_{0}^{u}$. The
optimal control problem is to minimize $J(u)$ over $\mathcal{U}[0,T]$, where
$\mathcal{U}[0,T]$ is the set of admissible controls. By the linearization and
weak convergence methods, the variational inequality and maximum principle
were obtained on a reference probability $P^{\ast}\in \mathcal{P}$ in
\cite{HJ1}, where $\mathcal{P}$ represents $G$-expectation $\mathbb{\hat{E}}$.

In this paper, we consider the following BSDE under a consistent convex
$\tilde{G}$-expectation $(\mathbb{\tilde{E}}_{t})_{t\geq0}$, which is
dominated by $G$-expectation $(\mathbb{\hat{E}}_{t})_{t\geq0}$,
\begin{equation}
Y_{t}^{u}=\mathbb{\tilde{E}}_{t}\left[  \Phi(X_{T}^{u})+\int_{t}^{T}%
f(s,X_{s}^{u},Y_{s}^{u},u_{s})ds+\sum_{i,j=1}^{d}\int_{t}^{T}g^{ij}%
(s,X_{s}^{u},Y_{s}^{u},u_{s})d\langle B^{i},B^{j}\rangle_{s}\right]  .
\label{new-e4}%
\end{equation}
The reasons for considering convex expectations have been explained in detail
in F\"{o}llmer and Schied \cite{FS}. In particular, the wellposedness of BSDE
under convex $\tilde{G}$-expectation containing $(Y,Z,K)$ remains open. Our
control system consists of (\ref{new-e1}) and (\ref{new-e4}). The cost
functional is defined by $J(u):=Y_{0}^{u}$, and the optimal control problem is
to minimize $J(u)$ over $\mathcal{U}[0,T]$.

In order to derive the variational equation for cost functional by weak
convergence method, we need the representation theorem for convex $\tilde{G}%
$-expectation which is not given in the literature. So we first obtain the
representation theorem for $\mathbb{\tilde{E}}$ by the separation theorem for
convex sets, and then further obtain the representation theorem for
$\mathbb{\tilde{E}}_{t}$. On the other hand, since the linear BSDE under
convex $\tilde{G}$-expectation has no explicit solution (see Remark 4.9 in
\cite{HSW}), the standard linearization method fails under convex $\tilde{G}%
$-expectation. To overcome this difficulty, we obtain two important properties
(Propositions \ref{new-pro-3} and \ref{new-pro-4}) about the representation of
consistent convex $\tilde{G}$-expectation. On this basis, the key to get the
variational equation for cost functional is to prove inequality (\ref{ne-25})
by applying the linearization method under a probability. By a new
discretization method, we get the inequality (\ref{ne-25}). Furthermore, we
obtain the variational inequality and maximum principle under a reference probability.

This paper is organized as follows. In Section 2, we recall some basic results
of $G$-expectation and give the formulation of the control problem under
convex $\tilde{G}$-expectation. In Section 3, we obtain the representation
theorems for $\tilde{G}$-expectation and conditional $\tilde{G}$-expectation.
In Section 4, we obtain the variational inequality and maximum principle for
our control problem. In Section 5, we study the linear quadratic control
problem by using the obtained maximum principle.

\section{Preliminaries and formulation of the control problem}

We recall some basic notions and results of $G$-expectation. The readers may
refer to \cite{P2019, HJPS1, HJPS} for more details.

For each fixed $T>0$, let $\Omega_{T}=C_{0}([0,T];\mathbb{R}^{d})$ be the
space of $\mathbb{R}^{d}$-valued continuous functions on $[0,T]$ with
$\omega_{0}=0$. The canonical process $B_{t}(\omega):=\omega_{t}$, for
$\omega \in \Omega_{T}$ and $t\in \lbrack0,T]$, denote $B=(B^{1},\ldots
,B^{d})^{T}$. For each given $t\in \lbrack0,T]$, set%
\[
Lip(\Omega_{t}):=\{ \varphi(B_{t_{1}},B_{t_{2}},\ldots,B_{t_{N}}):N\geq
1,t_{1}<\cdots<t_{N}\leq t,\varphi \in C_{b.Lip}(\mathbb{R}^{d\times N})\},
\]
where $C_{b.Lip}(\mathbb{R}^{d\times N})$ denotes the space of bounded
Lipschitz functions on $\mathbb{R}^{d\times N}$.

For each given monotonic and sublinear function $G:\mathbb{S}_{d}%
\rightarrow \mathbb{R}$, where $\mathbb{S}_{d}$ denotes the set of $d\times d$
symmetric matrices, Peng \cite{P07a, P08a} constructed a consistent sublinear
expectation space $(\Omega_{T},Lip(\Omega_{T}),\mathbb{\hat{E}},(\mathbb{\hat
{E}}_{t})_{t\in \lbrack0,T]})$, called $G$-expectation space. Moreover, for
each given $\tilde{G}:\mathbb{S}_{d}\rightarrow \mathbb{R}$ satisfying%
\[
\left \{
\begin{array}
[c]{l}%
\tilde{G}(0)=0,\\
\tilde{G}(A_{1})\leq \tilde{G}(A_{2})\text{ if }A_{1}\leq A_{2},\\
\tilde{G}(\alpha A_{1}+(1-\alpha)A_{2})\leq \alpha \tilde{G}(A_{1}%
)+(1-\alpha)\tilde{G}(A_{2})\text{ for }\alpha \in \lbrack0,1],\\
\tilde{G}(A_{1})-\tilde{G}(A_{2})\leq G(A_{1}-A_{2}),
\end{array}
\right.
\]
Peng \cite{P2019} constructed a consistent convex expectation space
$(\Omega_{T},Lip(\Omega_{T}),\mathbb{\tilde{E}},(\mathbb{\tilde{E}}_{t}%
)_{t\in \lbrack0,T]})$, called $\tilde{G}$-expectation space. For example,%
\[
\mathbb{\tilde{E}}_{t_{N-1}}[\varphi(B_{t_{1}},\ldots,B_{t_{N-1}},B_{t_{N}%
}-B_{t_{N-1}})]=\psi(B_{t_{1}},\ldots,B_{t_{N-1}}),
\]
where $\psi(x_{1},\ldots,x_{N-1})=u^{\varphi(x_{1},\ldots,x_{N-1},\cdot
)}(t_{N}-t_{N-1},0)$ and $u^{\varphi(x_{1},\ldots,x_{N-1},\cdot)}$ is the
viscosity solution (see \cite{CIP}) of the following PDE:
\[
\left \{
\begin{array}
[c]{l}%
\partial_{t}u-\tilde{G}(D_{x}^{2}u)=0,\\
u(0,x)=\varphi(x_{1},\ldots,x_{N-1},x).
\end{array}
\right.
\]
Specially, $\tilde{G}$-expectation is dominated by $G$-expectation in the
following sense:%
\begin{equation}
\mathbb{\tilde{E}}_{t}[X_{1}]-\mathbb{\tilde{E}}_{t}[X_{2}]\leq \mathbb{\hat
{E}}_{t}[X_{1}-X_{2}]\text{ for }X_{1},X_{2}\in Lip(\Omega_{T}),\text{ }%
t\in \lbrack0,T]. \label{ne-1}%
\end{equation}

For each $t\in \lbrack0,T]$ and $p\geq1$, denote by $L_{G}^{p}(\Omega_{t})$ the
completion of $Lip(\Omega_{t})$ under the norm $||X||_{L_{G}^{p}%
}:=(\mathbb{\hat{E}}[|X|^{p}])^{1/p}$. $\mathbb{\hat{E}}_{t}$ and
$\mathbb{\tilde{E}}_{t}$ can be continuously extended to $L_{G}^{1}(\Omega
_{T})$ under the norm $||\cdot||_{L_{G}^{1}}$, and the domination relation
(\ref{ne-1}) still holds.

The following result is the representation theorem for $G$-expectation.

\begin{theorem}
\label{th-re-1}(\cite{DHP11, HP09}) There exists a unique weakly compact and
convex set of probability measures $\mathcal{P}$ on $(\Omega_{T}%
,\mathcal{F}_{T})$ such that%
\[
\mathbb{\hat{E}}[X]=\sup_{P\in \mathcal{P}}E_{P}[X]\text{ for all }X\in
L_{G}^{1}(\Omega_{T}),
\]
where $\mathcal{F}_{t}=\mathcal{B}(\Omega_{t})=\sigma(\{B_{s}:s\leq t\})$ for
$t\leq T$.
\end{theorem}

The capacity associated with $\mathcal{P}$ is defined by%
\[
c(A):=\sup_{P\in \mathcal{P}}P(A)\text{ for }A\in \mathcal{F}_{T}.
\]
A set $A\in \mathcal{F}_{T}$ is polar if $c(A)=0$. A property holds
\textquotedblleft quasi-surely" (q.s. for short) if it holds outside a polar
set, and we do not distinguish two random variables $X$ and $Y$ if $X=Y$ q.s.

\begin{definition}
Let $M_{G}^{0}(0,T)$ be the space of simple processes in the following form:
for each integer $k\geq1$ and $0=t_{0}<\cdots<t_{k}=T$,%
\[
\eta_{t}=\sum_{i=0}^{k-1}\xi_{i}I_{[t_{i},t_{i+1})}(t),
\]
where $\xi_{i}\in Lip(\Omega_{t_{i}})$ for $i=0,1,\ldots,k-1$.
\end{definition}

For each fixed $p\geq1$, denote by $M_{G}^{p}(0,T)$ the completion of
$M_{G}^{0}(0,T)$ under the norm $||\eta||_{M_{G}^{p}}:=\left(  \mathbb{\hat
{E}}[\int_{0}^{T}|\eta_{t}|^{p}dt]\right)  ^{1/p}$. For each $\eta^{i}\in
M_{G}^{2}(0,T)$, $i=1,\ldots,d$, denote $\eta=(\eta^{1},\ldots,\eta^{d}%
)^{T}\in M_{G}^{2}(0,T;\mathbb{R}^{d})$, the $G$-It\^{o} integral $\int
_{0}^{T}\eta_{t}^{T}dB_{t}$ is well defined.

Now we give the formulation of the control problem. Let the control domain $U$
be a given nonempty convex set of $\mathbb{R}^{m}$. The set of admissible
controls is denoted by%
\[
\mathcal{U}[0,T]=M_{G}^{2}(0,T;U):=\{u\in M_{G}^{2}(0,T;\mathbb{R}^{m}%
):u_{t}\in U\text{ for }t\in \lbrack0,T]\}.
\]
Consider the followng forward-backward control system: for each $u\in
\mathcal{U}[0,T]$,%

\begin{equation}
\left \{
\begin{array}
[c]{rl}%
dX_{t}^{u}= & b(t,X_{t}^{u},u_{t})dt+\sum_{i,j=1}^{d}h^{ij}(t,X_{t}^{u}%
,u_{t})d\langle B^{i},B^{j}\rangle_{t}+\sum_{i=1}^{d}\sigma^{i}(t,X_{t}%
^{u},u_{t})dB_{t}^{i},\\
X_{0}^{u}= & x_{0}\in \mathbb{R}^{n},
\end{array}
\right.  \label{ne-2}%
\end{equation}%
\begin{equation}
Y_{t}^{u}=\mathbb{\tilde{E}}_{t}\left[  \Phi(X_{T}^{u})+\int_{t}^{T}%
f(s,X_{s}^{u},Y_{s}^{u},u_{s})ds+\sum_{i,j=1}^{d}\int_{t}^{T}g^{ij}%
(s,X_{s}^{u},Y_{s}^{u},u_{s})d\langle B^{i},B^{j}\rangle_{s}\right]  ,
\label{ne-3}%
\end{equation}
where $t\in \lbrack0,T]$, $\langle B\rangle=(\langle B^{i},B^{j}\rangle
)_{i,j=1}^{d}$ is the quadratic variation of $B$, for $i$, $j=1,\ldots,d,$
\[%
\begin{array}
[c]{l}%
b,h^{ij},\sigma^{i}:[0,T]\times \Omega_{T}\times \mathbb{R}^{n}\times
U\rightarrow \mathbb{R}^{n},\\
f,g^{ij}:[0,T]\times \Omega_{T}\times \mathbb{R}^{n}\times \mathbb{R}\times
U\rightarrow \mathbb{R},\\
\Phi:\Omega_{T}\times \mathbb{R}^{n}\rightarrow \mathbb{R}.
\end{array}
\]
In this paper, we need the following assumptions:

\begin{description}
\item[(A1)] For $(x,y,v)\in \mathbb{R}^{n}\times \mathbb{R}\times U$ and $i$,
$j=1,\ldots,d$, $(b(t,x,v))_{t\in \lbrack0,T]}$, $(h^{ij}(t,x,v))_{t\in
\lbrack0,T]}$, $(\sigma^{i}(t,x,v))_{t\in \lbrack0,T]}\in M_{G}^{2}%
(0,T;\mathbb{R}^{n})$, $(f(t,x,y,v))_{t\in \lbrack0,T]}$, $(g^{ij}%
(t,x,y,v))_{t\in \lbrack0,T]}\in M_{G}^{1}(0,T)$, $\Phi(x)\in L_{G}^{1}%
(\Omega_{T})$.

\item[(A2)] For $i$, $j=1,\ldots,d$, $b$, $h^{ij}$, $\sigma^{i}$, $f$,
$g^{ij}$ and $\Phi$ are differentiable in $(x,y,v)$.

\item[(A3)] For $i$, $j=1,\ldots,d$, the derivatives of $b$, $h^{ij}$,
$\sigma^{i}$ in $(x,v)$ and the derivatives of $f$, $g^{ij}$ in $y$ are bounded.

\item[(A4)] There exists a constant $L>0$ such that for $i$, $j=1,\ldots,d$,
$t\leq T$, $(x,y,v)\in \mathbb{R}^{n}\times \mathbb{R}\times U$,%
\[
(|f_{x}|+|f_{v}|+|g_{x}^{ij}|+|g_{v}^{ij}|)(t,x,y,v)+|\Phi_{x}(x)|\leq
L(1+|x|+|v|).
\]

\item[(A5)] There exists a modulus of continuity $\bar{\omega}:[0,\infty
)\rightarrow \lbrack0,\infty)$ such that for each $t\leq T$, $x_{1}$, $x_{2}%
\in \mathbb{R}^{n}$, $y_{1}$, $y_{2}\in \mathbb{R}$, $v_{1}$, $v_{2}\in U$,%
\[
|\phi(t,x_{1},y_{1},v_{1})-\phi(t,x_{2},y_{2},v_{2})|\leq \bar{\omega}%
(|x_{1}-x_{2}|+|y_{1}-y_{2}|+|v_{1}-v_{2}|),
\]
where $\phi$ is the derivative of $b$, $h^{ij}$, $\sigma^{i}$, $f$, $g^{ij}$,
$\Phi$ in $(x,y,v)$ for $i$, $j=1,\ldots,d$.
\end{description}

Under the assumptions (A1)-(A4), for each given $u\in \mathcal{U}[0,T]$, the
above forward-backward control system\ (\ref{ne-2})-(\ref{ne-3}) has a unique
solution $(X^{u},Y^{u})\in M_{G}^{2}(0,T;\mathbb{R}^{n})\times M_{G}^{1}(0,T)$
(see \cite{P2019}).

The cost functional is defined by%
\begin{equation}
J(u):=Y_{0}^{u}. \label{nne-4}%
\end{equation}
The related stochastic optimal control problem is to minimize $J(u)$ over
$\mathcal{U}[0,T]$. If there exists a $u^{\ast}\in \mathcal{U}[0,T]$ such that%
\begin{equation}
J(u^{\ast})=\min_{u\in \mathcal{U}[0,T]}J(u), \label{ne-4}%
\end{equation}
we call $u^{\ast}$ an optimal control. In order to obtain the maximum
principle, we first establish the representation theorem for $\tilde{G}$-expectation.

\section{Representation theorem for $\tilde{G}$-expectation}

Since $\mathcal{P}$ in Theorem \ref{th-re-1} is nondominated, i.e., we cannot
find a probability measure $Q$ on $(\Omega_{T},\mathcal{B}(\Omega_{T}))$ such
that $P\ll Q$ for any $P\in \mathcal{P}$, the representation theorem for convex
expectation $\mathbb{\tilde{E}}$ is different from that in \cite{FS}. For each
$P\in \mathcal{P}$ and $t\leq T$, define%
\begin{equation}
\alpha_{0}^{t}(P):=\sup_{Y\in L_{G}^{1}(\Omega_{t})}\left(  E_{P}%
[Y]-\mathbb{\tilde{E}}[Y]\right)  . \label{ne-5}%
\end{equation}
Since $E_{P}[0]-\mathbb{\tilde{E}}[0]=0$, we have $\alpha_{0}^{t}(P)\in
\lbrack0,\infty]$.

\begin{theorem}
\label{th-re-2}Let $\mathcal{P}$ be given in Theorem \ref{th-re-1}. Then, for
each fixed $t\leq T$, we have%
\begin{equation}
\mathbb{\tilde{E}}[X]=\sup_{P\in \mathcal{P}}\left(  E_{P}[X]-\alpha_{0}%
^{t}(P)\right)  ,\text{ }\forall X\in L_{G}^{1}(\Omega_{t}), \label{ne-8}%
\end{equation}
where $\alpha_{0}^{t}(P)$ is defined in (\ref{ne-5}).
\end{theorem}

\begin{proof}
For each fixed $X\in L_{G}^{1}(\Omega_{t})$, by the definition of $\alpha
_{0}^{t}(P)$, we have%
\[
\alpha_{0}^{t}(P)\geq E_{P}[X]-\mathbb{\tilde{E}}[X]\text{ for each }%
P\in \mathcal{P}\text{,}%
\]
which implies
\[
\mathbb{\tilde{E}}[X]\geq \sup_{P\in \mathcal{P}}\left(  E_{P}[X]-\alpha_{0}%
^{t}(P)\right)  .
\]
In the following, we show that there exists a $Q\in \mathcal{P}$ such that
$\mathbb{\tilde{E}}[X]=E_{Q}[X]-\alpha_{0}^{t}(Q)$.

The space $L_{G}^{1}(\Omega_{t})\times \mathbb{R}$ is a Banach space under the
norm $\mathbb{\hat{E}}[|Y|]+|a|$ for $Y\in L_{G}^{1}(\Omega_{t})$ and
$a\in \mathbb{R}$. It is easy to check that%
\[
A=\{(Y,\beta):Y\in L_{G}^{1}(\Omega_{t})\text{, }\beta \geq \mathbb{\tilde{E}%
}[Y]\}
\]
is a closed convex set and the interior of $A$ is
\[
A^{\circ}=\{(Y,\beta):Y\in L_{G}^{1}(\Omega_{t})\text{, }\beta>\mathbb{\tilde
{E}}[Y]\}.
\]
Since $A^{\circ}\cap \{(X,\mathbb{\tilde{E}}[X])\}=\emptyset$, by the
separation theorem for convex sets (see Theorem 3.7 in Chapter IV in
\cite{Con}), there exists a non-zero continuous linear functional $L$ on
$L_{G}^{1}(\Omega_{t})\times \mathbb{R}$ such that%
\begin{equation}
L(Y,\beta)\geq L(X,\mathbb{\tilde{E}}[X])\text{ for any }(Y,\beta)\in A.
\label{ne-6}%
\end{equation}
Note that $L(Y,\beta)=L(Y,0)+\beta L(0,1)$, $L(-Y,0)=-L(Y,0)$ and $\beta
\geq \mathbb{\tilde{E}}[Y]$, we deduce $L(0,1)>0$ by (\ref{ne-6}). Thus we
obtain%
\[
L(Y,\mathbb{\tilde{E}}[Y])=L(Y,0)+\mathbb{\tilde{E}}[Y]L(0,1)\geq
L(X,\mathbb{\tilde{E}}[X])\text{ for any }Y\in L_{G}^{1}(\Omega_{t}),
\]
which implies%
\[
\mathbb{\tilde{E}}[Y]\geq(L(0,1))^{-1}[-L(Y,0)+L(X,\mathbb{\tilde{E}%
}[X])]\text{ for any }Y\in L_{G}^{1}(\Omega_{t}).
\]
Set $b=(L(0,1))^{-1}L(X,\mathbb{\tilde{E}}[X])\in \mathbb{R}$ and define
$E:L_{G}^{1}(\Omega_{t})\rightarrow \mathbb{R}$ by%
\[
E[Y]:=-(L(0,1))^{-1}L(Y,0)\text{ for any }Y\in L_{G}^{1}(\Omega_{t}).
\]
It is clear that $E$ is a linear functional on $L_{G}^{1}(\Omega_{t})$ and%
\begin{equation}
\mathbb{\tilde{E}}[X]=E[X]+b\text{, }\mathbb{\tilde{E}}[Y]\geq E[Y]+b\text{
for any }Y\in L_{G}^{1}(\Omega_{t}). \label{ne-7}%
\end{equation}
Now we prove that $E$ is a linear expectation on $L_{G}^{1}(\Omega_{t})$ and
dominated by $\mathbb{\hat{E}}$. Since
\[
\lambda=\mathbb{\tilde{E}}[\lambda]\geq \lambda E[1]+b\text{ and }%
-\lambda=\mathbb{\tilde{E}}[-\lambda]\geq-\lambda E[1]+b
\]
for each $\lambda>0$, we get $E[1]=1$ by taking $\lambda \rightarrow \infty$.
For each $Y\geq0$, noting that $0\geq \mathbb{\tilde{E}}[-\lambda
Y]\geq-\lambda E[Y]+b$ for each $\lambda>0$, then we obtain $E[Y]\geq0$ by
taking $\lambda \rightarrow \infty$. Since $\mathbb{\tilde{E}}$ is dominated by
$\mathbb{\hat{E}}$, we get%
\[
\lambda \mathbb{\hat{E}}[Y]=\mathbb{\hat{E}}[\lambda Y]\geq \mathbb{\tilde{E}%
}[\lambda Y]\geq \lambda E[Y]+b
\]
for each $\lambda>0$ and $Y\in L_{G}^{1}(\Omega_{t})$, which implies
$E[Y]\leq \mathbb{\hat{E}}[Y]$ for each $Y\in L_{G}^{1}(\Omega_{t})$ by taking
$\lambda \rightarrow \infty$. Thus there exists a $Q\in \mathcal{P}$ such that
$E[Y]=E_{Q}[Y]$ for each $Y\in L_{G}^{1}(\Omega_{t})$. By (\ref{ne-5}) and
(\ref{ne-7}), it is easy to verify that $\alpha_{0}^{t}(Q)=-b$. Thus we obtain
(\ref{ne-8}).
\end{proof}

\begin{remark}
For each fixed $X\in L_{G}^{1}(\Omega_{t})$ with $t<T$. If $P_{1}%
\in \mathcal{P}$ satisfies $\mathbb{\tilde{E}}[X]=E_{P_{1}}[X]-\alpha_{0}%
^{t}(P_{1})$, then we cannot deduce $\alpha_{0}^{T}(P_{1})=\alpha_{0}%
^{t}(P_{1})$ by (\ref{ne-5}). But if $P_{2}\in \mathcal{P}$ satisfies
$\mathbb{\tilde{E}}[X]=E_{P_{2}}[X]-\alpha_{0}^{T}(P_{2})$, then we can deduce
$\alpha_{0}^{t}(P_{2})=\alpha_{0}^{T}(P_{2})$ by (\ref{ne-5}).
\end{remark}

Based on the above representation theorem, we have the following results.

\begin{proposition}
\label{pro-3}Let $\{X_{n}:n\geq1\}$ be a sequence in $L_{G}^{1}(\Omega_{T})$
such that $\mathbb{\hat{E}}[|X_{n}-X|]\rightarrow0$ as $n\rightarrow \infty$.
Set%
\[
\mathcal{P}^{\ast}=\{P\in \mathcal{P}:E_{P}[X]-\alpha_{0}^{T}(P)=\mathbb{\tilde
{E}}[X]\},
\]%
\[
\mathcal{P}_{n}^{\ast}=\{P\in \mathcal{P}:E_{P}[X_{n}]-\alpha_{0}%
^{T}(P)=\mathbb{\tilde{E}}[X_{n}]\}.
\]

\begin{description}
\item[(1)] For each given $P_{1}$, $P_{2}\in \mathcal{P}$ and $\lambda
\in \lbrack0,1]$, we have $\alpha_{0}^{T}(\lambda P_{1}+(1-\lambda)P_{2}%
)\leq \lambda \alpha_{0}^{T}(P_{1})+(1-\lambda)\alpha_{0}^{T}(P_{2})$.

\item[(2)] Let $P_{n}\in \mathcal{P}$, $n\geq1$, converge weakly to
$P\in \mathcal{P}$. Then $\alpha_{0}^{T}(P)\leq \underset{n\rightarrow \infty
}{\lim \inf}\alpha_{0}^{T}(P_{n})$.

\item[(3)] $\mathcal{P}^{\ast}$ is a nonempty weakly compact and convex set.

\item[(4)] Let $P_{n}\in \mathcal{P}_{n}^{\ast}$, $n\geq1$, converge weakly to
$P\in \mathcal{P}$. Then $P\in \mathcal{P}^{\ast}$.
\end{description}
\end{proposition}

\begin{proof}
(1) By the definition of $\alpha_{0}^{T}(\cdot)$ in (\ref{ne-5}), for any
$Y\in L_{G}^{1}(\Omega_{T})$, we have%
\[
E_{P^{\lambda}}[Y]-\mathbb{\tilde{E}}[Y]=\lambda(E_{P_{1}}[Y]-\mathbb{\tilde
{E}}[Y])+(1-\lambda)(E_{P_{2}}[Y]-\mathbb{\tilde{E}}[Y])\leq \lambda \alpha
_{0}^{T}(P_{1})+(1-\lambda)\alpha_{0}^{T}(P_{2}),
\]
where $P^{\lambda}=\lambda P_{1}+(1-\lambda)P_{2}$. Then we obtain%
\[
\alpha_{0}^{T}(\lambda P_{1}+(1-\lambda)P_{2})=\sup_{Y\in L_{G}^{1}(\Omega
_{T})}\left(  E_{P^{\lambda}}[Y]-\mathbb{\tilde{E}}[Y]\right)  \leq
\lambda \alpha_{0}^{T}(P_{1})+(1-\lambda)\alpha_{0}^{T}(P_{2}).
\]

(2) For any given $Y\in L_{G}^{1}(\Omega_{T})$, by Lemma 29 in \cite{DHP11}
and (\ref{ne-5}), we have%
\[
E_{P}[Y]-\mathbb{\tilde{E}}[Y]=\lim_{n\rightarrow \infty}\left(  E_{P_{n}%
}[Y]-\mathbb{\tilde{E}}[Y]\right)  \leq \underset{n\rightarrow \infty}{\lim \inf
}\alpha_{0}^{T}(P_{n}).
\]
Then we get%
\[
\alpha_{0}^{T}(P)=\sup_{Y\in L_{G}^{1}(\Omega_{T})}\left(  E_{P}%
[Y]-\mathbb{\tilde{E}}[Y]\right)  \leq \underset{n\rightarrow \infty}{\lim \inf
}\alpha(P_{n}).
\]

(3) By the proof of Theorem \ref{th-re-2}, we know that $\mathcal{P}^{\ast}$
is nonempty. By (1) and (2), it is easy to check that $\mathcal{P}^{\ast}$ is
weakly compact and convex.

(4) Since%
\[
|E_{P_{n}}[X_{n}]-E_{P}[X]|\leq E_{P_{n}}[|X_{n}-X|]+|E_{P_{n}}[X]-E_{P}%
[X]|\leq \mathbb{\hat{E}}[|X_{n}-X|]+|E_{P_{n}}[X]-E_{P}[X]|,
\]
we get $E_{P_{n}}[X_{n}]\rightarrow E_{P}[X]$ as $n\rightarrow \infty$. It
follows from $P_{n}\in \mathcal{P}_{n}^{\ast}$ and (2) that $E_{P}%
[X]-\alpha_{0}^{T}(P)\geq \mathbb{\tilde{E}}[X]$. By Theorem \ref{th-re-2}, we
obtain $P\in \mathcal{P}^{\ast}$.
\end{proof}

Now we study the representation theorem for conditional $\tilde{G}%
$-expectation $\mathbb{\tilde{E}}_{t}$. For each fixed $P\in \mathcal{P}$, set%
\begin{equation}
\mathcal{P}(t,P)=\{Q\in \mathcal{P}:Q(A)=P(A)\text{ for any }A\in
\mathcal{F}_{t}\}, \label{ne-9}%
\end{equation}
and define%
\begin{equation}
\alpha_{t}^{T}(P):=\underset{Y\in L_{G}^{1}(\Omega_{T})}{ess\sup}^{P}%
(E_{P}[Y|\mathcal{F}_{t}]-\mathbb{\tilde{E}}_{t}[Y]),\text{ }P\text{-a.s.,}
\label{ne-10}%
\end{equation}
where the $ess\sup$ is obtained under the probability $P$.

\begin{theorem}
\label{new-th-re-1}For each given $P\in \mathcal{P}$, let $\mathcal{P}(t,P)$ be
defined in (\ref{ne-9}). Then, for each $X\in L_{G}^{1}(\Omega_{T})$, we have%
\begin{equation}
\mathbb{\tilde{E}}_{t}[X]=\underset{Q\in \mathcal{P}(t,P)}{ess\sup}^{P}%
(E_{Q}[X|\mathcal{F}_{t}]-\alpha_{t}^{T}(Q)),\text{ }P\text{-a.s.,}
\label{ne-13}%
\end{equation}
where $\alpha_{t}^{T}(Q)$ is defined in (\ref{ne-10}).
\end{theorem}

\begin{proof}
The proof is divided into three steps.

Step 1: For each fixed $X=\phi(B_{t_{1}}^{t},\ldots,B_{t_{N}}^{t})$, where
$t<t_{1}<\cdots<t_{N}\leq T$, $\phi \in C_{b.Lip}(\mathbb{R}^{d\times N})$ and
$B_{s}^{t}=B_{s}-B_{t}$ for $s\geq t$. By the proof of Theorem \ref{th-re-2},
there exists a $Q_{1}\in \mathcal{P}$ such that $\mathbb{\tilde{E}}%
[X]=E_{Q_{1}}[X]-\alpha_{0}^{T}(Q_{1})$. For each $Y\in Lip(\Omega_{T})$, we
can find $0<s_{1}<\cdots<s_{i}\leq t$, $t<t_{1}<\cdots<t_{j}\leq T$ and
$\varphi \in C_{b.Lip}(\mathbb{R}^{d\times(i+j)})$ such that
\begin{equation}
Y=\varphi(B_{s_{1}},\ldots,B_{s_{i}},B_{t_{1}}^{t},\ldots,B_{t_{j}}^{t}).
\label{ne-11}%
\end{equation}
By Lemma 18 in \cite{HP21}, we can find a $Q^{\ast}\in \mathcal{P}(t,P)$ such
that, for each $Y$ in (\ref{ne-11}),%
\begin{equation}
E_{Q^{\ast}}[Y|\mathcal{F}_{t}]=\psi(B_{s_{1}},\ldots,B_{s_{i}}),\text{ where
}\psi(x_{1},\ldots,x_{i})=E_{Q_{1}}[\varphi(x_{1},\ldots,x_{i},B_{t_{1}}%
^{t},\ldots,B_{t_{j}}^{t})]. \label{ne-12}%
\end{equation}
Thus
\[
\underset{Y\in L_{G}^{1}(\Omega_{T})}{ess\sup}^{P}(E_{Q^{\ast}}[Y|\mathcal{F}%
_{t}]-\mathbb{\tilde{E}}_{t}[Y])=\underset{Y\in Lip(\Omega_{T})}{ess\sup}%
^{P}(E_{Q^{\ast}}[Y|\mathcal{F}_{t}]-\mathbb{\tilde{E}}_{t}[Y])\leq \alpha
_{0}^{T}(Q_{1}),\text{ }P\text{-a.s.}%
\]
Since $E_{Q^{\ast}}[X|\mathcal{F}_{t}]=E_{Q_{1}}[X]$ and $\mathbb{\tilde{E}%
}_{t}[X]=\mathbb{\tilde{E}}[X]$, we obtain $\alpha_{t}^{T}(Q^{\ast}%
)=\alpha_{0}^{T}(Q_{1})$, $P$-a.s. Then we get $\mathbb{\tilde{E}}%
_{t}[X]=E_{Q^{\ast}}[X|\mathcal{F}_{t}]-\alpha_{t}^{T}(Q^{\ast})$. By
(\ref{ne-10}), it is easy to verify that%
\[
\mathbb{\tilde{E}}_{t}[X]\geq \underset{Q\in \mathcal{P}(t,P)}{ess\sup}%
^{P}(E_{Q}[X|\mathcal{F}_{t}]-\alpha_{t}^{T}(Q)),\text{ }P\text{-a.s.}%
\]
Thus we obtain (\ref{ne-13}) for this $X$.

Step 2: For each fixed $X=\phi(B_{s_{1}},\ldots,B_{s_{k}},B_{t_{1}}^{t}%
,\ldots,B_{t_{N}}^{t})$, where $0<s_{1}<\cdots<s_{k}\leq t$, $t<t_{1}%
<\cdots<t_{N}\leq T$ and $\phi \in C_{b.Lip}(\mathbb{R}^{d\times(k+N)})$. For
each $(x_{1},\ldots,x_{k})\in \mathbb{R}^{d\times k}$, define%
\[
L(x_{1},\ldots,x_{k})=\underset{Q\in \mathcal{P}(t,P)}{ess\sup}^{P}(E_{Q}%
[\phi(x_{1},\ldots,x_{k},B_{t_{1}}^{t},\ldots,B_{t_{N}}^{t})|\mathcal{F}%
_{t}]-\alpha_{t}^{T}(Q)),\text{ }P\text{-a.s.}%
\]
By the similar proof of Lemma 19 in \cite{HP21}, we can get%
\[
\underset{Q\in \mathcal{P}(t,P)}{ess\sup}^{P}(E_{Q}[X|\mathcal{F}_{t}%
]-\alpha_{t}^{T}(Q))=L(B_{s_{1}},\ldots,B_{s_{k}}),\text{ }P\text{-a.s.}%
\]
It follows from Step 1 that $L(x_{1},\ldots,x_{k})=\mathbb{\tilde{E}}%
[\phi(x_{1},\ldots,x_{k},B_{t_{1}}^{t},\ldots,B_{t_{N}}^{t})]$. By the
definition of $\mathbb{\tilde{E}}_{t}$, we obtain $\mathbb{\tilde{E}}%
_{t}[X]=L(B_{s_{1}},\ldots,B_{s_{k}})$. Thus we obtain (\ref{ne-13}) for this
$X$.

Step 3: For each fixed $X\in L_{G}^{1}(\Omega_{T})$. There exists a squence
$X_{n}\in Lip(\Omega_{T})$, $n\geq1$, such that $\mathbb{\hat{E}}%
[|X_{n}-X|]\rightarrow0$ as $n\rightarrow \infty$. Note that%
\[
|\mathbb{\tilde{E}}_{t}[X_{n}]-\mathbb{\tilde{E}}_{t}[X]|\leq \mathbb{\hat{E}%
}_{t}[|X_{n}-X|],
\]%
\begin{align*}
&  \left \vert \underset{Q\in \mathcal{P}(t,P)}{ess\sup}^{P}(E_{Q}%
[X_{n}|\mathcal{F}_{t}]-\alpha_{t}^{T}(Q))-\underset{Q\in \mathcal{P}%
(t,P)}{ess\sup}^{P}(E_{Q}[X|\mathcal{F}_{t}]-\alpha_{t}^{T}(Q))\right \vert \\
&  \leq \underset{Q\in \mathcal{P}(t,P)}{ess\sup}^{P}E_{Q}[|X_{n}-X||\mathcal{F}%
_{t}]\\
&  \leq \mathbb{\hat{E}}_{t}[|X_{n}-X|],\text{ }P\text{-a.s.}%
\end{align*}
Thus we obtain (\ref{ne-13}) for this $X$ by Step 2. The proof is complete.
\end{proof}

The following two propositions are important to derive the maximum principle.

\begin{proposition}
\label{new-pro-3}Let $P\in \mathcal{P}$ be given. Then we have%
\begin{equation}
E_{P}[\alpha_{t}^{T}(P)]=\sup_{X\in L_{G}^{1}(\Omega_{T})}E_{P}\left[
E_{P}[X|\mathcal{F}_{t}]-\mathbb{\tilde{E}}_{t}[X]\right]  , \label{ne-14}%
\end{equation}%
\begin{equation}
\alpha_{0}^{T}(P)=\alpha_{0}^{t}(P)+E_{P}[\alpha_{t}^{T}(P)]. \label{ne-15}%
\end{equation}

\end{proposition}

\begin{proof}
By the property of $ess\sup$, we know that there exists a sequence $X_{n}\in
L_{G}^{1}(\Omega_{T})$, $n\geq1$, such that%
\[
\alpha_{t}^{T}(P)=\sup_{n\geq1}(E_{P}[X_{n}|\mathcal{F}_{t}]-\mathbb{\tilde
{E}}_{t}[X_{n}]),\text{ }P\text{-a.s.}%
\]
Set $Y_{n}=\vee_{i=1}^{n}(E_{P}[X_{i}|\mathcal{F}_{t}]-\mathbb{\tilde{E}}%
_{t}[X_{i}])$, $n\geq1$. It is clear that $Y_{n}\uparrow \alpha_{t}^{T}(P)$ as
$n\rightarrow \infty$, $P$-a.s. For each fixed $n$, set%
\[
A_{i}=\{E_{P}[X_{i}|\mathcal{F}_{t}]-\mathbb{\tilde{E}}_{t}[X_{i}]=Y_{n}\}
\backslash \cup_{j=1}^{i-1}A_{j},\text{ }i=1,\ldots,n,
\]
where $\cup_{j=1}^{0}A_{j}=\emptyset$. Then $Y_{n}=\sum_{i=1}^{n}(E_{P}%
[X_{i}|\mathcal{F}_{t}]-\mathbb{\tilde{E}}_{t}[X_{i}])I_{A_{i}}$. For each
$k\geq1$, there exist compact sets $B_{i}^{k}\subset A_{i}$ and $B_{i}^{k}%
\in \mathcal{F}_{t}$, $i=1$,$\ldots$,$n$, such that $P(A_{i}\backslash
B_{i}^{k})<k^{-1}$. By Tietze's extension theorem, there exist continuous
functions $f_{i}^{k}\in L_{G}^{1}(\Omega_{t})$, $i=1$,$\ldots$,$n$, such that%
\[
0\leq f_{i}^{k}\leq1\text{, }f_{i}^{k}|_{B_{i}^{k}}=1\text{ and }\{f_{i}%
^{k}\not =0\} \cap \{f_{j}^{k}\not =0\}=\emptyset \text{ for }i\not =j\text{.}%
\]
Set $L^{k}=\sum_{i=1}^{n}f_{i}^{k}X_{i}$. It is obvious that $L^{k}\in
L_{G}^{1}(\Omega_{T})$. Note that%
\[
|E_{P}[L^{k}|\mathcal{F}_{t}]-E_{P}[X_{i}|\mathcal{F}_{t}]|I_{A_{i}}%
\leq \mathbb{\hat{E}}_{t}\left[  \sum_{j=1}^{n}|X_{j}|\right]  I_{A_{i}%
\backslash B_{i}^{k}},
\]%
\[
|\mathbb{\tilde{E}}_{t}[L^{k}]-\mathbb{\tilde{E}}_{t}[X_{i}]|I_{A_{i}}%
\leq \mathbb{\hat{E}}_{t}\left[  \sum_{j=1}^{n}|X_{j}|\right]  I_{A_{i}%
\backslash B_{i}^{k}}.
\]
Then we obtain
\[
|E_{P}[L^{k}|\mathcal{F}_{t}]-\mathbb{\tilde{E}}_{t}[L^{k}]-Y_{n}%
|\leq2\mathbb{\hat{E}}_{t}\left[  \sum_{j=1}^{n}|X_{j}|\right]  \sum_{i=1}%
^{n}I_{A_{i}\backslash B_{i}^{k}}.
\]
Taking $k\rightarrow \infty$, we get
\[
E_{P}[Y_{n}]\leq \sup_{X\in L_{G}^{1}(\Omega_{T})}E_{P}\left[  E_{P}%
[X|\mathcal{F}_{t}]-\mathbb{\tilde{E}}_{t}[X]\right]  .
\]
Since $Y_{n}\uparrow \alpha_{t}^{T}(P)$, $P$-a.s., we deduce%
\[
E_{P}[\alpha_{t}^{T}(P)]\leq \sup_{X\in L_{G}^{1}(\Omega_{T})}E_{P}\left[
E_{P}[X|\mathcal{F}_{t}]-\mathbb{\tilde{E}}_{t}[X]\right]  .
\]
By the definition of $\alpha_{t}^{T}(P)$, it is easy to get%
\[
E_{P}[\alpha_{t}^{T}(P)]\geq \sup_{X\in L_{G}^{1}(\Omega_{T})}E_{P}\left[
E_{P}[X|\mathcal{F}_{t}]-\mathbb{\tilde{E}}_{t}[X]\right]  .
\]
Thus we obtain (\ref{ne-14}). By the definition of $\alpha_{0}^{T}(P)$,
$\alpha_{0}^{t}(P)$ and (\ref{ne-14}), we have%
\begin{align*}
\alpha_{0}^{T}(P)  &  =\sup \{E_{P}[X]-\mathbb{\tilde{E}}[X]:X\in L_{G}%
^{1}(\Omega_{T})\} \\
&  =\sup \{E_{P}[X+Y_{1}]-\mathbb{\tilde{E}}[X+Y_{1}]:X\in L_{G}^{1}(\Omega
_{T}),Y_{1}\in L_{G}^{1}(\Omega_{t})\} \\
&  =\sup \{E_{P}[E_{P}[X|\mathcal{F}_{t}]-\mathbb{\tilde{E}}_{t}[X]]+E_{P}%
[\mathbb{\tilde{E}}_{t}[X]+Y_{1}]-\mathbb{\tilde{E}}[\mathbb{\tilde{E}}%
_{t}[X]+Y_{1}]:X\in L_{G}^{1}(\Omega_{T}),Y_{1}\in L_{G}^{1}(\Omega_{t})\} \\
&  =\sup_{X\in L_{G}^{1}(\Omega_{T})}E_{P}\left[  E_{P}[X|\mathcal{F}%
_{t}]-\mathbb{\tilde{E}}_{t}[X]\right]  +\sup_{Y\in L_{G}^{1}(\Omega_{t}%
)}(E_{P}[Y]-\mathbb{\tilde{E}}[Y])\\
&  =E_{P}[\alpha_{t}^{T}(P)]+\alpha_{0}^{t}(P).
\end{align*}
The proof is complete.
\end{proof}

\begin{remark}
\label{new-re-4}Similar to the proof of Proposition \ref{new-pro-3}, we can
further obtain%
\[
\alpha_{t}^{T}(P)=\alpha_{t}^{s}(P)+E_{P}[\alpha_{s}^{T}(P)|\mathcal{F}%
_{t}]\text{ for }t\leq s,\text{ }P\text{-a.s., }%
\]
where%
\[
\alpha_{t}^{s}(P):=\underset{Y\in L_{G}^{1}(\Omega_{s})}{ess\sup}^{P}%
(E_{P}[Y|\mathcal{F}_{t}]-\mathbb{\tilde{E}}_{t}[Y]),\text{ }P\text{-a.s.}%
\]

\end{remark}

\begin{proposition}
\label{new-pro-4}Let $X\in L_{G}^{1}(\Omega_{T})$ be given and let
$P\in \mathcal{P}$ satisfy $\mathbb{\tilde{E}}[X]=E_{P}[X]-\alpha_{0}^{T}(P)$.
Then we have%
\begin{equation}
\mathbb{\tilde{E}}_{t}[X]=E_{P}[X|\mathcal{F}_{t}]-\alpha_{t}^{T}(P),\text{
}P\text{-a.s.,} \label{ne-16}%
\end{equation}%
\begin{equation}
\mathbb{\tilde{E}}\left[  \mathbb{\tilde{E}}_{t}[X]\right]  =E_{P}\left[
\mathbb{\tilde{E}}_{t}[X]\right]  -\alpha_{0}^{t}(P). \label{ne-17}%
\end{equation}

\end{proposition}

\begin{proof}
It follows from Theorems \ref{new-th-re-1} and \ref{th-re-2} that
\[
\mathbb{\tilde{E}}_{t}[X]\geq E_{P}[X|\mathcal{F}_{t}]-\alpha_{t}%
^{T}(P),\text{ }P\text{-a.s.,}%
\]%
\[
\mathbb{\tilde{E}}[X]=\mathbb{\tilde{E}}\left[  \mathbb{\tilde{E}}%
_{t}[X]\right]  \geq E_{P}\left[  \mathbb{\tilde{E}}_{t}[X]\right]
-\alpha_{0}^{t}(P).
\]
If (\ref{ne-16}) and (\ref{ne-17}) donot hold at the same time, then we get%
\[
\mathbb{\tilde{E}}[X]>E_{P}[X]-\alpha_{0}^{t}(P)-E_{P}[\alpha_{t}^{T}(P)].
\]
By Proposition \ref{new-pro-3}, we obtain $\mathbb{\tilde{E}}[X]>E_{P}%
[X]-\alpha_{0}^{T}(P)$, which contradicts the assumption. Thus we obtain
(\ref{ne-16}) and (\ref{ne-17}).
\end{proof}

\section{Maximum principle}

\subsection{Variational equation for cost functional}

Let $u^{\ast}\in \mathcal{U}[0,T]$ be an optimal control for control problem
(\ref{ne-4}) and let $(X^{\ast},Y^{\ast})$ be the corresponding solution for
(\ref{ne-2})-(\ref{ne-3}). For each fixed $u\in \mathcal{U}[0,T]$, set
\[
u^{\varepsilon}=(1-\varepsilon)u^{\ast}+\varepsilon u=u^{\ast}+\varepsilon
(u-u^{\ast})\text{ for }\varepsilon \in \lbrack0,1].
\]
Since $U$ is convex, we have $u^{\varepsilon}\in \mathcal{U}[0,T]$. The
solution for (\ref{ne-2})-(\ref{ne-3}) corresponding to $u^{\varepsilon}$ is
denoted by $(X^{\varepsilon},Y^{\varepsilon})$. In the following, the constant
$C$ will change from line to line in the proof.

Set
\[
b_{x}(\cdot)=\left[
\begin{array}
[c]{ccc}%
b_{1x_{1}}(\cdot) & \cdots & b_{1x_{n}}(\cdot)\\
\vdots &  & \vdots \\
b_{nx_{1}}(\cdot) & \cdots & b_{nx_{n}}(\cdot)
\end{array}
\right]  .
\]
Under this notation, $f_{x}(\cdot)=(f_{x_{1}}(\cdot),\ldots,f_{x_{n}}%
(\cdot))\in \mathbb{R}^{1\times n}$. The following variational equation for
$G$-SDE (\ref{ne-2}) was obtained in \cite{HJ1}.

\begin{proposition}
\label{pro-4}(\cite{HJ1})Let assumptions (A1)-(A5) hold. Then%
\[
\sup_{t\leq T}\mathbb{\hat{E}}\left[  |\tilde{X}_{t}^{\varepsilon}%
|^{2}\right]  \leq C\text{ and }\lim_{\varepsilon \downarrow0}\sup_{t\leq
T}\mathbb{\hat{E}}\left[  |\tilde{X}_{t}^{\varepsilon}|^{2}\right]  =0,
\]
where $\tilde{X}_{t}^{\varepsilon}=\varepsilon^{-1}(X_{t}^{\varepsilon}%
-X_{t}^{\ast})-\hat{X}_{t}$, the constant $C>0$ is independent of
$\varepsilon$,%
\[
\left \{
\begin{array}
[c]{rl}%
d\hat{X}_{t}= & [b_{x}(t)\hat{X}_{t}+b_{v}(t)(u_{t}-u_{t}^{\ast}%
)]dt+\sum_{i,j=1}^{d}[h_{x}^{ij}(t)\hat{X}_{t}+h_{v}^{ij}(t)(u_{t}-u_{t}%
^{\ast})]d\langle B^{i},B^{j}\rangle_{t}\\
& +\sum_{i=1}^{d}[\sigma_{x}^{i}(t)\hat{X}_{t}+\sigma_{v}^{i}(t)(u_{t}%
-u_{t}^{\ast})]dB_{t}^{i},\\
\hat{X}_{0}= & 0,
\end{array}
\right.
\]
$b_{x}(t)=b_{x}(t,X_{t}^{\ast},u_{t}^{\ast})$, similar definition for
$b_{v}(t)$, $h_{x}^{ij}(t)$, $h_{v}^{ij}(t)$, $\sigma_{x}^{i}(t)$ and
$\sigma_{v}^{i}(t)$.
\end{proposition}

In order to obtain the variational equation for cost functional, we need the
following estimates. For simplicity of presentation, we use the following
notations: for $\alpha \in \lbrack0,1]$,
\[
\varphi_{l}(s,\alpha)=\varphi_{l}(s,X_{s}^{\ast}+\alpha(X_{s}^{\varepsilon
}-X_{s}^{\ast}),Y_{s}^{\ast}+\alpha(Y_{s}^{\varepsilon}-Y_{s}^{\ast}%
),u_{s}^{\ast}+\alpha \varepsilon(u_{s}-u_{s}^{\ast}))\text{,}%
\]%
\[
\varphi_{l}(s)=\varphi_{l}(s,X_{s}^{\ast},Y_{s}^{\ast},u_{s}^{\ast}),
\]
where $\varphi=f$, $g^{ij}$ and $l=x$, $y$, $v$. Set%
\[
\Lambda_{t}^{\varepsilon}=\exp \left(  \int_{0}^{t}\left(  \int_{0}^{1}%
f_{y}(s,\alpha)d\alpha \right)  ds+\sum_{i,j=1}^{d}\int_{0}^{t}\left(  \int
_{0}^{1}g_{y}^{ij}(s,\alpha)d\alpha \right)  d\langle B^{i},B^{j}\rangle
_{s}\right)  ,
\]%
\begin{equation}
\Lambda_{t}=\exp \left(  \int_{0}^{t}f_{y}(s)ds+\sum_{i,j=1}^{d}\int_{0}%
^{t}g_{y}^{ij}(s)d\langle B^{i},B^{j}\rangle_{s}\right)  . \label{ne-19}%
\end{equation}
Since $f_{y}$, $g_{y}^{ij}$ are bounded and $d\langle B^{i}\rangle_{s}\leq
Cds$ for $i$, $j=1$,$\ldots$,$d$, we have $\Lambda_{t}^{\varepsilon}%
+\Lambda_{t}\leq C$ for $t\in \lbrack0,T]$.

\begin{lemma}
\label{le-5}Let assumptions (A1)-(A5) hold. Then

\begin{description}
\item[(1)] $\sup_{t\leq T}\mathbb{\hat{E}}[|Y_{t}^{\varepsilon}-Y_{t}^{\ast
}|]\leq C\varepsilon$, where $C$ is independent of $\varepsilon$.

\item[(2)] $\mathbb{\hat{E}}\left[  |\Phi_{x}(X_{T}^{\ast})\hat{X}%
_{T}||\Lambda_{T}^{\varepsilon}-\Lambda_{T}|+\int_{0}^{T}\Pi_{t}|\Lambda
_{t}^{\varepsilon}-\Lambda_{t}|dt\right]  \rightarrow0$ as $\varepsilon
\rightarrow0$, where
\[
\Pi_{t}=|f_{x}(t)\hat{X}_{t}|+|f_{v}(t)(u_{t}-u_{t}^{\ast})|+\sum_{i,j=1}%
^{d}(|g_{x}^{ij}(t)\hat{X}_{t}|+|g_{v}^{ij}(t)(u_{t}-u_{t}^{\ast})|).
\]

\end{description}
\end{lemma}

\begin{proof}
By (\ref{ne-1}) and (A4), we get%
\[
|Y_{t}^{\varepsilon}-Y_{t}^{\ast}|\leq C\mathbb{\hat{E}}_{t}\left[
\xi^{\varepsilon}|X_{T}^{\varepsilon}-X_{T}^{\ast}|+\int_{t}^{T}%
[|Y_{s}^{\varepsilon}-Y_{s}^{\ast}|+\eta_{s}^{\varepsilon}(|X_{s}%
^{\varepsilon}-X_{s}^{\ast}|+\varepsilon|u_{s}-u_{s}^{\ast}|)]ds\right]  ,
\]
where $\xi^{\varepsilon}=1+|X_{T}^{\varepsilon}|+|X_{T}^{\ast}|$, $\eta
_{s}^{\varepsilon}=1+|X_{s}^{\varepsilon}|+|X_{s}^{\ast}|+|u_{s}|+|u_{s}%
^{\ast}|$, the constant $C>0$ is independent of $\varepsilon$. By the Gronwall
inequality (see Theorem 3.10 in \cite{HJPS}), we obtain%
\[
|Y_{t}^{\varepsilon}-Y_{t}^{\ast}|\leq C\mathbb{\hat{E}}_{t}\left[
\xi^{\varepsilon}|X_{T}^{\varepsilon}-X_{T}^{\ast}|+\int_{t}^{T}\eta
_{s}^{\varepsilon}(|X_{s}^{\varepsilon}-X_{s}^{\ast}|+\varepsilon|u_{s}%
-u_{s}^{\ast}|)ds\right]  .
\]
Thus%
\begin{align*}
&  \sup_{t\leq T}\mathbb{\hat{E}}[|Y_{t}^{\varepsilon}-Y_{t}^{\ast}|]\\
&  \leq C\mathbb{\hat{E}}\left[  \xi^{\varepsilon}|X_{T}^{\varepsilon}%
-X_{T}^{\ast}|+\int_{0}^{T}\eta_{s}^{\varepsilon}(|X_{s}^{\varepsilon}%
-X_{s}^{\ast}|+\varepsilon|u_{s}-u_{s}^{\ast}|)ds\right] \\
&  \leq C\left[  (\mathbb{\hat{E}}[|\xi^{\varepsilon}|^{2}])^{1/2}%
(\mathbb{\hat{E}}[|X_{T}^{\varepsilon}-X_{T}^{\ast}|^{2}])^{1/2}+\sqrt{2}%
\int_{0}^{T}(\mathbb{\hat{E}}[|\eta_{s}^{\varepsilon}|^{2}])^{1/2}%
(\mathbb{\hat{E}}[|X_{s}^{\varepsilon}-X_{s}^{\ast}|^{2}+\varepsilon^{2}%
|u_{s}-u_{s}^{\ast}|^{2}])^{1/2}ds\right]  .
\end{align*}
By Proposition \ref{pro-4}, we deduce%
\[
\sup_{t\leq T}\mathbb{\hat{E}}[|X_{t}^{\varepsilon}|^{2}]\leq C\text{ and
}\sup_{t\leq T}\mathbb{\hat{E}}[|X_{t}^{\varepsilon}-X_{t}^{\ast}|^{2}]\leq
C\varepsilon^{2},
\]
where the constant $C>0$ is independent of $\varepsilon$. Thus we obtain (1).

Since $\Lambda_{t}^{\varepsilon}+\Lambda_{t}\leq C$ for $t\in \lbrack0,T]$ and%
\[
\lim_{\lambda \rightarrow \infty}\mathbb{\hat{E}}\left[  (|\Phi_{x}(X_{T}^{\ast
})\hat{X}_{T}|-|\Phi_{x}(X_{T}^{\ast})\hat{X}_{T}|\wedge \lambda)+\int_{0}%
^{T}(\Pi_{t}-\Pi_{t}\wedge \lambda)dt\right]  =0,
\]
(2) is equivalent to%
\begin{equation}
\lim_{\varepsilon \rightarrow0}\mathbb{\hat{E}}\left[  |\Lambda_{T}%
^{\varepsilon}-\Lambda_{T}|+\int_{0}^{T}|\Lambda_{t}^{\varepsilon}-\Lambda
_{t}|dt\right]  =0. \label{ne-18}%
\end{equation}
It follows from $\Lambda^{\varepsilon}+\Lambda \leq C$ that
\[
|\Lambda_{t}^{\varepsilon}-\Lambda_{t}|\leq C\left(  \int_{0}^{t}\int_{0}%
^{1}|f_{y}(s,\alpha)-f_{y}(s)|d\alpha ds+\sum_{i,j=1}^{d}\int_{0}^{t}\int
_{0}^{1}|g_{y}^{ij}(s,\alpha)-g_{y}^{ij}(s)|d\alpha ds\right)  ,
\]
where the constant $C>0$ is independent of $\varepsilon$. For each fixed
$N>1$, set%
\[
A_{s}=\{|X_{s}^{\varepsilon}-X_{s}^{\ast}|\leq N\varepsilon \} \text{, }%
C_{s}=\{|Y_{s}^{\varepsilon}-Y_{s}^{\ast}|\leq N\varepsilon \} \text{, }%
D_{s}=\{|u_{s}-u_{s}^{\ast}|\leq N\}.
\]
By (A5), we obtain%
\[
|\Lambda_{t}^{\varepsilon}-\Lambda_{t}|\leq C\left(  \bar{\omega
}(3N\varepsilon)+\int_{0}^{t}(I_{A_{s}^{c}}+I_{C_{s}^{c}}+I_{D_{s}^{c}%
})ds\right)  .
\]
Thus by (1) we get%
\begin{align*}
\mathbb{\hat{E}}[|\Lambda_{t}^{\varepsilon}-\Lambda_{t}|]  &  \leq C\left(
\bar{\omega}(3N\varepsilon)+\int_{0}^{t}N^{-1}\mathbb{\hat{E}}[\varepsilon
^{-1}|X_{s}^{\varepsilon}-X_{s}^{\ast}|+\varepsilon^{-1}|Y_{s}^{\varepsilon
}-Y_{s}^{\ast}|+|u_{s}-u_{s}^{\ast}|]ds\right) \\
&  \leq C(\bar{\omega}(3N\varepsilon)+N^{-1}),
\end{align*}
where the constant $C>0$ is independent of $\varepsilon$ and $N$. Taking
$\varepsilon \rightarrow0$ and then $N\rightarrow \infty$, we obtain
(\ref{ne-18}), which implies (2).
\end{proof}

The following estimates are essentially Lemma 4.5 of \cite{HJ1}, the only
difference in the proof is that we use the estimate (1) of $|Y_{t}%
^{\varepsilon}-Y_{t}^{\ast}|$ in Lemma \ref{le-5}.

\begin{lemma}
\label{le-6}(\cite{HJ1})Let assumptions (A1)-(A5) hold. Then

\begin{description}
\item[(1)] $\mathbb{\hat{E}}[|I_{1}^{\varepsilon}|]=o(\varepsilon)$, where
$I_{1}^{\varepsilon}=\Phi(X_{T}^{\varepsilon})-\Phi(X_{T}^{\ast}%
)-\varepsilon \Phi_{x}(X_{T}^{\ast})\hat{X}_{T}$.

\item[(2)] $\mathbb{\hat{E}}\left[  \int_{0}^{T}\left(  |I_{2}^{\varepsilon
}(s)|+\sum_{i,j=1}^{d}|I_{ij}^{\varepsilon}(s)|\right)  ds\right]
=o(\varepsilon)$, where
\[
I_{2}^{\varepsilon}(s)=\int_{0}^{1}f_{x}(s,\alpha)d\alpha(X_{s}^{\varepsilon
}-X_{s}^{\ast})-\varepsilon f_{x}(s)\hat{X}_{s}+\varepsilon \int_{0}^{1}%
(f_{v}(s,\alpha)-f_{v}(s))d\alpha(u_{s}-u_{s}^{\ast}),
\]%
\[
I_{ij}^{\varepsilon}(s)=\int_{0}^{1}g_{x}^{ij}(s,\alpha)d\alpha(X_{s}%
^{\varepsilon}-X_{s}^{\ast})-\varepsilon g_{x}^{ij}(s)\hat{X}_{s}%
+\varepsilon \int_{0}^{1}(g_{v}^{ij}(s,\alpha)-g_{v}^{ij}(s))d\alpha
(u_{s}-u_{s}^{\ast}).
\]

\end{description}
\end{lemma}

Now we give the variational equation for cost functional. Set%
\begin{equation}
\mathcal{P}^{\ast}=\{P\in \mathcal{P}:E_{P}[\xi^{\ast}]-\alpha_{0}%
^{T}(P)=\mathbb{\tilde{E}}[\xi^{\ast}]\}, \label{ne-20}%
\end{equation}%
\begin{equation}
\mathcal{P}_{\varepsilon}^{\ast}=\{P\in \mathcal{P}:E_{P}[\xi^{\varepsilon
}]-\alpha_{0}^{T}(P)=\mathbb{\tilde{E}}[\xi^{\varepsilon}]\}, \label{ne-21}%
\end{equation}
where%
\[
\xi^{\ast}=\Phi(X_{T}^{\ast})+\int_{0}^{T}f(s,X_{s}^{\ast},Y_{s}^{\ast}%
,u_{s}^{\ast})ds+\sum_{i,j=1}^{d}\int_{0}^{T}g^{ij}(s,X_{s}^{\ast},Y_{s}%
^{\ast},u_{s}^{\ast})d\langle B^{i},B^{j}\rangle_{s},
\]%
\[
\xi^{\varepsilon}=\Phi(X_{T}^{\varepsilon})+\int_{0}^{T}f(s,X_{s}%
^{\varepsilon},Y_{s}^{\varepsilon},u_{s}^{\varepsilon})ds+\sum_{i,j=1}^{d}%
\int_{0}^{T}g^{ij}(s,X_{s}^{\varepsilon},Y_{s}^{\varepsilon},u_{s}%
^{\varepsilon})d\langle B^{i},B^{j}\rangle_{s}.
\]

\begin{theorem}
\label{th-7}Let assumptions (A1)-(A5) hold and let $u^{\ast}\in \mathcal{U}%
[0,T]$ be an optimal control for control problem (\ref{ne-4}). Then, for each
$u\in \mathcal{U}[0,T]$, there exists a $P^{u}\in \mathcal{P}^{\ast}$ satisfying%
\[
\lim_{\varepsilon \rightarrow0}\frac{J(u^{\varepsilon})-J(u^{\ast}%
)}{\varepsilon}=\sup_{P\in \mathcal{P}^{\ast}}E_{P}[L^{u}]=E_{P^{u}}[L^{u}],
\]
where $(\Lambda_{t})_{t\leq T}$ is defined in (\ref{ne-19}) and%
\[
L^{u}=\Phi_{x}(X_{T}^{\ast})\hat{X}_{T}\Lambda_{T}+\int_{0}^{T}[f_{x}%
(s)\hat{X}_{s}+f_{v}(s)(u_{s}-u_{s}^{\ast})]\Lambda_{s}ds+\sum_{i,j=1}^{d}%
\int_{0}^{T}[g_{x}^{ij}(s)\hat{X}_{s}+g_{v}^{ij}(s)(u_{s}-u_{s}^{\ast
})]\Lambda_{s}d\langle B^{i},B^{j}\rangle_{s}.
\]

\end{theorem}

\begin{proof}
For each $P\in \mathcal{P}^{\ast}$, by Proposition \ref{new-pro-4} and Theorem
\ref{new-th-re-1}, we have%
\begin{equation}
\mathbb{\tilde{E}}_{t}[\xi^{\ast}]=E_{P}[\xi^{\ast}|\mathcal{F}_{t}%
]-\alpha_{t}^{T}(P),\text{ }\mathbb{\tilde{E}}_{t}[\xi^{\varepsilon}]\geq
E_{P}[\xi^{\varepsilon}|\mathcal{F}_{t}]-\alpha_{t}^{T}(P),\text{
}P\text{-a.s.} \label{ne-23}%
\end{equation}
Note that%
\begin{align*}
&  f(s,X_{s}^{\varepsilon},Y_{s}^{\varepsilon},u_{s}^{\varepsilon}%
)-f(s,X_{s}^{\ast},Y_{s}^{\ast},u_{s}^{\ast})\\
&  =\int_{0}^{1}f_{x}(s,\alpha)d\alpha(X_{s}^{\varepsilon}-X_{s}^{\ast}%
)+\int_{0}^{1}f_{y}(s,\alpha)d\alpha(Y_{s}^{\varepsilon}-Y_{s}^{\ast
})+\varepsilon \int_{0}^{1}f_{v}(s,\alpha)d\alpha(u_{s}-u_{s}^{\ast}).
\end{align*}
Then we get $P$-a.s.%
\[
\hat{Y}_{t}^{\varepsilon}\geq E_{P}\left[  \left.  I_{1}^{\varepsilon
}+\varepsilon \tilde{L}_{t}^{u}+\int_{t}^{T}\left[  \tilde{f}_{y}(s)\hat{Y}%
_{s}^{\varepsilon}+I_{2}^{\varepsilon}(s)\right]  ds+\sum_{i,j=1}^{d}\int
_{t}^{T}\left[  \tilde{g}_{y}^{ij}(s)\hat{Y}_{s}^{\varepsilon}+I_{ij}%
^{\varepsilon}(s)\right]  d\langle B^{i},B^{j}\rangle_{s}\right \vert
\mathcal{F}_{t}\right]  ,
\]
where $\hat{Y}_{t}^{\varepsilon}=Y_{t}^{\varepsilon}-Y_{t}^{\ast}$, $\tilde
{f}_{y}(s)=\int_{0}^{1}f_{y}(s,\alpha)d\alpha$, $\tilde{g}_{y}^{ij}%
(s)=\int_{0}^{1}g_{y}^{ij}(s,\alpha)d\alpha$, $I_{1}^{\varepsilon}$,
$I_{2}^{\varepsilon}(s)$ and $I_{ij}^{\varepsilon}(s)$ are defined in Lemma
\ref{le-6},
\[
\tilde{L}_{t}^{u}=\Phi_{x}(X_{T}^{\ast})\hat{X}_{T}+\int_{t}^{T}[f_{x}%
(s)\hat{X}_{s}+f_{v}(s)(u_{s}-u_{s}^{\ast})]ds+\sum_{i,j=1}^{d}\int_{t}%
^{T}[g_{x}^{ij}(s)\hat{X}_{s}+g_{v}^{ij}(s)(u_{s}-u_{s}^{\ast})]d\langle
B^{i},B^{j}\rangle_{s}.
\]
For $t\in \lbrack0,T]$, set%
\begin{equation}
\delta_{t}=\mathbb{\tilde{E}}_{t}[\xi^{\varepsilon}]-\mathbb{\tilde{E}}%
_{t}[\xi^{\ast}]-E_{P}[\xi^{\varepsilon}-\xi^{\ast}|\mathcal{F}_{t}%
]\geq0,\text{ }P\text{-a.s.} \label{ne-22}%
\end{equation}
Then we obtain $P$-a.s.%
\[
\hat{Y}_{t}^{\varepsilon}-\delta_{t}=E_{P}\left[  \left.  I_{1}^{\varepsilon
}+\varepsilon \tilde{L}_{t}^{u}+\int_{t}^{T}\left[  \tilde{f}_{y}(s)\hat{Y}%
_{s}^{\varepsilon}+I_{2}^{\varepsilon}(s)\right]  ds+\sum_{i,j=1}^{d}\int
_{t}^{T}\left[  \tilde{g}_{y}^{ij}(s)\hat{Y}_{s}^{\varepsilon}+I_{ij}%
^{\varepsilon}(s)\right]  d\langle B^{i},B^{j}\rangle_{s}\right \vert
\mathcal{F}_{t}\right]  .
\]
By the same proof as Lemma 4.6 in \cite{HSW}, we get%
\begin{equation}
\hat{Y}_{0}^{\varepsilon}=\delta_{0}+E_{P}\left[  I_{1}^{\varepsilon}%
\Lambda_{T}^{\varepsilon}+\varepsilon L_{\varepsilon}^{u}+\int_{0}^{T}\left[
\tilde{f}_{y}(s)\delta_{s}+I_{2}^{\varepsilon}(s)\right]  \Lambda
_{s}^{\varepsilon}ds+\sum_{i,j=1}^{d}\int_{0}^{T}\left[  \tilde{g}_{y}%
^{ij}(s)\delta_{s}+I_{ij}^{\varepsilon}(s)\right]  \Lambda_{s}^{\varepsilon
}d\langle B^{i},B^{j}\rangle_{s}\right]  , \label{ne-24}%
\end{equation}
where%
\[
L_{\varepsilon}^{u}=\Phi_{x}(X_{T}^{\ast})\hat{X}_{T}\Lambda_{T}^{\varepsilon
}+\int_{0}^{T}[f_{x}(s)\hat{X}_{s}+f_{v}(s)(u_{s}-u_{s}^{\ast})]\Lambda
_{s}^{\varepsilon}ds+\sum_{i,j=1}^{d}\int_{0}^{T}[g_{x}^{ij}(s)\hat{X}%
_{s}+g_{v}^{ij}(s)(u_{s}-u_{s}^{\ast})]\Lambda_{s}^{\varepsilon}d\langle
B^{i},B^{j}\rangle_{s}.
\]

Now we prove that
\begin{equation}
\delta_{0}+E_{P}\left[  \int_{0}^{T}\tilde{f}_{y}(s)\delta_{s}\Lambda
_{s}^{\varepsilon}ds+\sum_{i,j=1}^{d}\int_{0}^{T}\tilde{g}_{y}^{ij}%
(s)\delta_{s}\Lambda_{s}^{\varepsilon}d\langle B^{i},B^{j}\rangle_{s}\right]
\geq0. \label{ne-25}%
\end{equation}
For each $N\geq1$, set $t_{k}^{N}=N^{-1}kT$ for $k=0$,$\ldots$,$N$, and
$\delta_{s}^{N}=\sum_{k=1}^{N}\delta_{t_{k}^{N}}I_{(t_{k-1}^{N},t_{k}^{N}%
]}(s)$. By Lemma 2.16 in \cite{HSW} and $(E_{P}[\xi^{\varepsilon}-\xi^{\ast
}|\mathcal{F}_{t}])_{t\leq T}$ is uniformly integrable under $P$, we obtain%
\begin{align*}
&  \lim_{N\rightarrow \infty}E_{P}\left[  \int_{0}^{T}\tilde{f}_{y}%
(s)\delta_{s}^{N}\Lambda_{s}^{\varepsilon}ds+\sum_{i,j=1}^{d}\int_{0}%
^{T}\tilde{g}_{y}^{ij}(s)\delta_{s}^{N}\Lambda_{s}^{\varepsilon}d\langle
B^{i},B^{j}\rangle_{s}\right] \\
&  =E_{P}\left[  \int_{0}^{T}\tilde{f}_{y}(s)\delta_{s}\Lambda_{s}%
^{\varepsilon}ds+\sum_{i,j=1}^{d}\int_{0}^{T}\tilde{g}_{y}^{ij}(s)\delta
_{s}\Lambda_{s}^{\varepsilon}d\langle B^{i},B^{j}\rangle_{s}\right]  .
\end{align*}
Since $\int_{t_{k-1}^{N}}^{t_{k}^{N}}\tilde{f}_{y}(s)\Lambda_{s}^{\varepsilon
}ds+\sum_{i,j=1}^{d}\int_{t_{k-1}^{N}}^{t_{k}^{N}}\tilde{g}_{y}^{ij}%
(s)\Lambda_{s}^{\varepsilon}d\langle B^{i},B^{j}\rangle_{s}=\Lambda_{t_{k}%
^{N}}^{\varepsilon}-\Lambda_{t_{k-1}^{N}}^{\varepsilon}$ and $\delta_{T}=0$,
we get%
\begin{align*}
&  \delta_{0}+E_{P}\left[  \int_{0}^{T}\tilde{f}_{y}(s)\delta_{s}^{N}%
\Lambda_{s}^{\varepsilon}ds+\sum_{i,j=1}^{d}\int_{0}^{T}\tilde{g}_{y}%
^{ij}(s)\delta_{s}^{N}\Lambda_{s}^{\varepsilon}d\langle B^{i},B^{j}\rangle
_{s}\right] \\
&  =\delta_{0}+E_{P}\left[  \sum_{k=1}^{N}\delta_{t_{k}^{N}}\left(
\Lambda_{t_{k}^{N}}^{\varepsilon}-\Lambda_{t_{k-1}^{N}}^{\varepsilon}\right)
\right] \\
&  =\sum_{k=0}^{N-1}E_{P}\left[  \Lambda_{t_{k}^{N}}^{\varepsilon}\left(
\delta_{t_{k}^{N}}-E_{P}[\delta_{t_{k+1}^{N}}|\mathcal{F}_{t_{k}^{N}}]\right)
\right]  .
\end{align*}
By (\ref{ne-23}) and (\ref{ne-22}), we have $\delta_{t}=\mathbb{\tilde{E}}%
_{t}[\xi^{\varepsilon}]-E_{P}[\xi^{\varepsilon}|\mathcal{F}_{t}]+\alpha
_{t}^{T}(P)$, $P$-a.s. It follows from Remark \ref{new-re-4} that%
\[
\delta_{t_{k}^{N}}-E_{P}[\delta_{t_{k+1}^{N}}|\mathcal{F}_{t_{k}^{N}}%
]=\alpha_{t_{k}^{N}}^{t_{k+1}^{N}}(P)-\left(  E_{P}[\mathbb{\tilde{E}%
}_{t_{k+1}^{N}}[\xi^{\varepsilon}]|\mathcal{F}_{t_{k}^{N}}]-\mathbb{\tilde{E}%
}_{t_{k}^{N}}[\mathbb{\tilde{E}}_{t_{k+1}^{N}}[\xi^{\varepsilon}]]\right)
\geq0,\text{ }P\text{-a.s.}%
\]
Thus we obtain $E_{P}\left[  \Lambda_{t_{k}^{N}}^{\varepsilon}\left(
\delta_{t_{k}^{N}}-E_{P}[\delta_{t_{k+1}^{N}}|\mathcal{F}_{t_{k}^{N}}]\right)
\right]  \geq0$ for $k\leq N-1$, which implies (\ref{ne-25}) by taking
$N\rightarrow \infty$. By (\ref{ne-24}), (\ref{ne-25}), Lemmas \ref{le-5} and
\ref{le-6}, we get%
\begin{equation}
J(u^{\varepsilon})-J(u^{\ast})=\hat{Y}_{0}^{\varepsilon}\geq \varepsilon
E_{P}[L^{u}]+o(\varepsilon). \label{ne-26}%
\end{equation}
Thus we obtain%
\begin{equation}
\underset{\varepsilon \rightarrow0}{\lim \inf}\frac{J(u^{\varepsilon}%
)-J(u^{\ast})}{\varepsilon}\geq \sup_{P\in \mathcal{P}^{\ast}}E_{P}[L^{u}].
\label{ne-27}%
\end{equation}

Since $\mathcal{P}$ is weakly compact, we can choose $\varepsilon
_{k}\rightarrow0$ and $P_{k}\in \mathcal{P}_{\varepsilon_{k}}^{\ast}$ such that
$(P_{k})_{k\geq1}$ converges weakly to $P^{u}\in \mathcal{P}$ and%
\[
\underset{\varepsilon \rightarrow0}{\lim \sup}\frac{J(u^{\varepsilon}%
)-J(u^{\ast})}{\varepsilon}=\lim_{k\rightarrow \infty}\frac{J(u^{\varepsilon
_{k}})-J(u^{\ast})}{\varepsilon_{k}}.
\]
By (4) of Proposition \ref{pro-3}, we have $P^{u}\in \mathcal{P}^{\ast}$. For
$P_{k}$, by using the same analysis as in (\ref{ne-26}), we can deduce%
\[
J(u^{\varepsilon_{k}})-J(u^{\ast})\leq \varepsilon_{k}E_{P_{k}}[L^{u}%
]+o(\varepsilon_{k}).
\]
Since $L^{u}\in L_{G}^{1}(\Omega_{T})$, we have $E_{P_{k}}[L^{u}]\rightarrow
E_{P^{u}}[L^{u}]$. Thus we obtain%
\begin{equation}
\underset{\varepsilon \rightarrow0}{\lim \sup}\frac{J(u^{\varepsilon}%
)-J(u^{\ast})}{\varepsilon}\leq E_{P^{u}}[L^{u}]\leq \sup_{P\in \mathcal{P}%
^{\ast}}E_{P}[L^{u}]. \label{ne-28}%
\end{equation}
By (\ref{ne-27}) and (\ref{ne-28}), we obtain the desired result.
\end{proof}

\subsection{Maximum principle}

We first use Theorem \ref{th-7} to obtain the following variational inequality.

\begin{proposition}
\label{pro-8}Let assumptions (A1)-(A5) hold and let $u^{\ast}\in
\mathcal{U}[0,T]$ be an optimal control for control problem (\ref{ne-4}). Then
there exists a $P^{\ast}\in \mathcal{P}^{\ast}$ satisfying%
\[
\inf_{u\in \mathcal{U}[0,T]}E_{P^{\ast}}[L^{u}]\geq0,
\]
where $\mathcal{P}^{\ast}$ is defined in (\ref{ne-20}) and $L^{u}$ is defined
in Theorem \ref{th-7}.
\end{proposition}

\begin{proof}
By Theorem \ref{th-7}, we have%
\[
\inf_{u\in \mathcal{U}[0,T]}\sup_{P\in \mathcal{P}^{\ast}}E_{P}[L^{u}]\geq0.
\]
It follows from Proposition \ref{pro-3} that $\mathcal{P}^{\ast}$ is weakly
compact and convex. By Sion's minimax theorem, we get%
\[
\sup_{P\in \mathcal{P}^{\ast}}\inf_{u\in \mathcal{U}[0,T]}E_{P}[L^{u}%
]=\inf_{u\in \mathcal{U}[0,T]}\sup_{P\in \mathcal{P}^{\ast}}E_{P}[L^{u}]\geq0.
\]
Thus we can find a sequence $(P_{k})_{k\geq1}\subset \mathcal{P}^{\ast}$ such
that $(P_{k})_{k\geq1}$ converges weakly to $P^{\ast}\in \mathcal{P}^{\ast}$
and%
\[
\inf_{u\in \mathcal{U}[0,T]}E_{P_{k}}[L^{u}]\geq-k^{-1}.
\]
Since $L^{u}\in L_{G}^{1}(\Omega_{T})$, we have $E_{P_{k}}[L^{u}]\rightarrow
E_{P^{\ast}}[L^{u}]$. Thus we obtain the desired result.
\end{proof}

Now we introduce the following adjoint equation under $P^{\ast}$.%
\begin{equation}
\left \{
\begin{array}
[c]{rl}%
dp_{t}= & -[(b_{x}(t))^{T}p_{t}+f_{y}(t)p_{t}+(f_{x}(t))^{T}]dt\\
& -\sum_{i,j=1}^{d}[(h_{x}^{ij}(t))^{T}p_{t}+g_{y}^{ij}(t)p_{t}+(\sigma
_{x}^{i}(t))^{T}q_{t}^{j}+(g_{x}^{ij}(t))^{T}]d\langle B^{i},B^{j}\rangle
_{t}\\
& +\sum_{i=1}^{d}q_{t}^{i}dB_{t}^{i}+dN_{t},\\
p_{T}= & (\Phi_{x}(X_{T}^{\ast}))^{T}.
\end{array}
\right.  \label{ne-29}%
\end{equation}
It follows from El Karoui and Huang \cite{EH} that the adjoint equation
(\ref{ne-29}) has a unique solution $(p,q,N)\in M_{P^{\ast}}^{2}%
(0,T;\mathbb{R}^{n})\times M_{P^{\ast}}^{2}(0,T;\mathbb{R}^{n\times d})\times
M_{P^{\ast}}^{2,\bot}(0,T;\mathbb{R}^{n})$, where $q=[q^{1},\ldots,q^{d}]$ and%
\[
M_{P^{\ast}}^{2}(0,T;\mathbb{R}^{n})=\left \{  \eta:\eta \text{ is }%
\mathbb{R}^{n}\text{-valued progressively measurable and }E_{P^{\ast}}\left[
\int_{0}^{T}|\eta_{t}|^{2}dt\right]  <\infty \right \}  ,
\]%
\[
M_{P^{\ast}}^{2,\bot}(0,T;\mathbb{R}^{n})=\left \{  N:N\text{ is }%
\mathbb{R}^{n}\text{-valued square integrable martingale and orthogonal to
}B\right \}  .
\]
Define the Hamiltonian $H:[0,T]\times \Omega_{T}\times \mathbb{R}^{n}%
\times \mathbb{R}\times U\times \mathbb{R}^{n}\times \mathbb{R}^{n\times
d}\rightarrow \mathbb{R}$ as follows:%
\[
H(t,x,y,v,p,q):=p^{T}b(t,x,v)+f(t,x,y,v)+\sum_{i,j=1}^{d}[p^{T}h^{ij}%
(t,x,v)+(q^{j})^{T}\sigma^{i}(t,x,v)+g^{ij}(t,x,y,v)]\gamma_{t}^{ij},
\]
where $d\langle B^{i},B^{j}\rangle_{t}=\gamma_{t}^{ij}dt$. The following
theorem is the maximum principle.

\begin{theorem}
\label{th-9}Let assumptions (A1)-(A5) hold and let $u^{\ast}\in \mathcal{U}%
[0,T]$ be an optimal control for control problem (\ref{ne-4}). Then there
exists a $P^{\ast}\in \mathcal{P}^{\ast}$ satisfying%
\[
H_{v}(t,X_{t}^{\ast},Y_{t}^{\ast},u_{t}^{\ast},p_{t},q_{t})(v-u_{t}^{\ast
})\geq0\text{, }\forall v\in U\text{, a.e., }P^{\ast}\text{-a.s.,}%
\]
where $(X^{\ast},Y^{\ast})$ is the solution of (\ref{ne-2})-(\ref{ne-3})
corresponding to $u^{\ast}$, $\mathcal{P}^{\ast}$ is defined in (\ref{ne-20}),
$(p,q,N)$ is the solution of the adjoint equation (\ref{ne-29}) under
$P^{\ast}$.
\end{theorem}

\begin{proof}
By Proposition \ref{pro-8}, there exists a $P^{\ast}\in \mathcal{P}^{\ast}$
such that $E_{P^{\ast}}[L^{u}]\geq0$ for each $u\in \mathcal{U}[0,T]$. Applying
It\^{o}'s formula to $\Lambda_{t}\langle \hat{X}_{t},p_{t}\rangle$ under
$P^{\ast}$, we can deduce%
\[
E_{P^{\ast}}[L^{u}]=E_{P^{\ast}}\left[  \int_{0}^{T}H_{v}(t,X_{t}^{\ast}%
,Y_{t}^{\ast},u_{t}^{\ast},p_{t},q_{t})(u_{t}-u_{t}^{\ast})\Lambda
_{t}dt\right]  .
\]
Thus we only need to prove that $\mathcal{U}[0,T]=M_{G}^{2}(0,T;U)$ is dense
in $M_{P^{\ast}}^{2}(0,T;U)$. Since%
\[
\lim_{k\rightarrow \infty}E_{P^{\ast}}\left[  \int_{0}^{T}(|\eta_{t}|-|\eta
_{t}|\wedge k)^{2}dt\right]  =0\text{ for each }\eta \in M_{P^{\ast}}%
^{2}(0,T;U),
\]
we only need to prove the case that $U$ is bounded. Let $v_{0}\in$ri$(U)$ be
fixed, where ri$(U)$ is the relative interior of $U$. For each $\delta
\in \lbrack0,1)$, set%
\[
U_{\delta}=\{(1-\alpha)v_{0}+\alpha v:\forall \alpha \in \lbrack0,\delta],\text{
}\forall v\in \text{cl}(U)\backslash \text{ri}(U)\},
\]
where cl$(U)$ is the closure of $U$. By Theorem 6.1 in \cite{Ro}, we have
$U_{\delta}\subset U$. It is easy to check that $U_{\delta}$ is a closed
convex set and $\sup_{v_{1}\in U}\inf_{v_{2}\in U_{\delta}}|v_{1}-v_{2}|\leq
C(1-\delta)$, where the constant $C>0$ depends only on $U$. For each given
$\eta \in M_{P^{\ast}}^{2}(0,T;U)$, by Proposition 4.3 in \cite{HWZ} and
Lusin's theorem under $P^{\ast}$, there exists a sequence $(\eta^{k})_{k\geq
1}\subset M_{G}^{2}(0,T;\mathbb{R}^{m})$ such that $E_{P^{\ast}}\left[
\int_{0}^{T}|\eta_{t}-\eta_{t}^{k}|^{2}dt\right]  \rightarrow0$ as
$k\rightarrow \infty$. Set $\tilde{\eta}^{k}=P_{U_{\delta_{k}}}(\eta^{k})$,
where $\delta_{k}=1-k^{-1}$ and $P_{U_{\delta_{k}}}$ is the projection onto
$U_{\delta_{k}}$. Since $|P_{U_{\delta_{k}}}(v_{1})-P_{U_{\delta_{k}}}%
(v_{2})|\leq|v_{1}-v_{2}|$, we obtain that $(\tilde{\eta}^{k})_{k\geq1}%
\subset \mathcal{U}[0,T]$ and%
\begin{align*}
E_{P^{\ast}}\left[  \int_{0}^{T}|\eta_{t}-\tilde{\eta}_{t}^{k}|^{2}dt\right]
&  \leq2E_{P^{\ast}}\left[  \int_{0}^{T}(|\eta_{t}-P_{U_{\delta_{k}}}(\eta
_{t})|^{2}+|P_{U_{\delta_{k}}}(\eta_{t})-\tilde{\eta}_{t}^{k}|^{2})dt\right]
\\
&  \leq Ck^{-2}+2E_{P^{\ast}}\left[  \int_{0}^{T}|\eta_{t}-\eta_{t}^{k}%
|^{2}dt\right]  ,
\end{align*}
which implies $E_{P^{\ast}}\left[  \int_{0}^{T}|\eta_{t}-\tilde{\eta}_{t}%
^{k}|^{2}dt\right]  \rightarrow0$ as $k\rightarrow \infty$. Thus we obtain the
desired result.
\end{proof}

\subsection{Sufficient condition}

We give the following sufficient condition for optimality.

\begin{theorem}
\label{th-10}Let assumptions (A1)-(A5) hold. Assume that $u^{\ast}%
\in \mathcal{U}[0,T]$ and $P^{\ast}\in \mathcal{P}^{\ast}$ satisfy%
\[
H_{v}(t,X_{t}^{\ast},Y_{t}^{\ast},u_{t}^{\ast},p_{t},q_{t})(v-u_{t}^{\ast
})\geq0\text{, }\forall v\in U\text{, a.e., }P^{\ast}\text{-a.s.,}%
\]
where $\mathcal{P}^{\ast}$ is defined in (\ref{ne-20}), $(X^{\ast},Y^{\ast})$
is the solution of (\ref{ne-2})-(\ref{ne-3}) corresponding to $u^{\ast}$,
$(p,q,N)$ is the solution of (\ref{ne-29}) under $P^{\ast}$. Moreover, we
assume that $H(\cdot)$ is convex with respect to $x$, $y$, $v$ and $\Phi
(\cdot)$ is convex with respect to $x$. Then $u^{\ast}$ is an optimal control
for control problem (\ref{ne-4}).
\end{theorem}

\begin{proof}
For each given $u\in \mathcal{U}[0,T]$, the solution for (\ref{ne-2}%
)-(\ref{ne-3}) corresponding to $u$ is denoted by $(X,Y)$. Set $\tilde{X}%
_{t}=X_{t}-X_{t}^{\ast}$, $\tilde{Y}_{t}=Y_{t}-Y_{t}^{\ast}$ for $t\in
\lbrack0,T]$. We have%
\[
\left \{
\begin{array}
[c]{rl}%
d\tilde{X}_{t}= & [b_{x}(t)\tilde{X}_{t}+\beta_{1}(t)]dt+\sum_{i,j=1}%
^{d}[h_{x}^{ij}(t)\tilde{X}_{t}+\beta_{2}^{ij}(t)]d\langle B^{i},B^{j}%
\rangle_{t}+\sum_{i=1}^{d}[\sigma_{x}^{i}(t)\tilde{X}_{t}+\beta_{3}%
^{i}(t)]dB_{t}^{i},\\
\tilde{X}_{0}= & 0,
\end{array}
\right.
\]
where%
\[%
\begin{array}
[c]{rl}%
\beta_{1}(t)= & b(t,X_{t},u_{t})-b(t,X_{t}^{\ast},u_{t}^{\ast})-b_{x}%
(t)\tilde{X}_{t},\\
\beta_{2}^{ij}(t)= & h^{ij}(t,X_{t},u_{t})-h^{ij}(t,X_{t}^{\ast},u_{t}^{\ast
})-h_{x}^{ij}(t)\tilde{X}_{t},\\
\beta_{3}^{i}(t)= & \sigma^{i}(t,X_{t},u_{t})-\sigma^{i}(t,X_{t}^{\ast}%
,u_{t}^{\ast})-\sigma_{x}^{i}(t)\tilde{X}_{t}.
\end{array}
\]
Applying It\^{o}'s formula to $\Lambda_{t}\langle \tilde{X}_{t},p_{t}\rangle$
under $P^{\ast}$, where $(\Lambda_{t})_{t\leq T}$ is defined in (\ref{ne-19}),
we can get%
\begin{equation}%
\begin{array}
[c]{cl}%
E_{P^{\ast}}[\Phi_{x}(X_{T}^{\ast})\tilde{X}_{T}\Lambda_{T}]= & E_{P^{\ast}%
}\left[  \int_{0}^{T}(p_{t}^{T}\beta_{1}(t)-f_{x}(t)\tilde{X}_{t})\Lambda
_{t}dt\right] \\
& +\sum_{i,j=1}^{d}E_{P^{\ast}}\left[  \int_{0}^{T}(p_{t}^{T}\beta_{2}%
^{ij}(t)+(q_{t}^{j})^{T}\beta_{3}^{i}(t)-g_{x}^{ij}(t)\tilde{X}_{t}%
)\Lambda_{t}d\langle B^{i},B^{j}\rangle_{t}\right]  .
\end{array}
\label{ne-30}%
\end{equation}
Denote $\beta_{4}(T)=\Phi(X_{T})-\Phi(X_{T}^{\ast})-\Phi_{x}(X_{T}^{\ast
})\tilde{X}_{T}$ and%
\[%
\begin{array}
[c]{rl}%
\beta_{5}(t)= & f(t,X_{t},Y_{t},u_{t})-f(t,X_{t}^{\ast},Y_{t}^{\ast}%
,u_{t}^{\ast})-f_{x}(t)\tilde{X}_{t}-f_{y}(t)\tilde{Y}_{t},\\
\beta_{6}^{ij}(t)= & g^{ij}(t,X_{t},Y_{t},u_{t})-g^{ij}(t,X_{t}^{\ast}%
,Y_{t}^{\ast},u_{t}^{\ast})-g_{x}^{ij}(t)\tilde{X}_{t}-g_{y}^{ij}(t)\tilde
{Y}_{t}.
\end{array}
\]
The same analysis as in the proof of Theorem \ref{th-7}, we can obtain%
\begin{equation}%
\begin{array}
[c]{cl}%
\tilde{Y}_{0}\geq & E_{P^{\ast}}\left[  \Phi_{x}(X_{T}^{\ast})\tilde{X}%
_{T}\Lambda_{T}+\beta_{4}(T)\Lambda_{T}+\int_{0}^{T}(f_{x}(t)\tilde{X}%
_{t}+\beta_{5}(t))\Lambda_{t}dt\right] \\
& +\sum_{i,j=1}^{d}E_{P^{\ast}}\left[  \int_{0}^{T}(g_{x}^{ij}(t)\tilde{X}%
_{t}+\beta_{6}^{ij}(t))\Lambda_{t}d\langle B^{i},B^{j}\rangle_{t}\right]  .
\end{array}
\label{ne-31}%
\end{equation}
Note that $H_{v}(t,X_{t}^{\ast},Y_{t}^{\ast},u_{t}^{\ast},p_{t},q_{t}%
)(u_{t}-u_{t}^{\ast})\geq0$, a.e., $P^{\ast}$-a.s., $H(\cdot)$ is convex with
respect to $x$, $y$, $v$ and $\Phi(\cdot)$ is convex with respect to $x$. Then
we have $\beta_{4}(T)\geq0$ and
\begin{align*}
&  p_{t}^{T}\beta_{1}(t)+\beta_{5}(t)+\sum_{i,j=1}^{d}\left[  p_{t}^{T}%
\beta_{2}^{ij}(t)+(q_{t}^{j})^{T}\beta_{3}^{i}(t)+\beta_{6}^{ij}(t)\right]
\gamma_{t}^{ij}\\
&  =H(t,X_{t},Y_{t},u_{t},p_{t},q_{t})-H(t)-H_{x}(t)\tilde{X}_{t}%
-H_{y}(t)\tilde{Y}_{t}\\
&  \geq H(t,X_{t},Y_{t},u_{t},p_{t},q_{t})-H(t)-H_{x}(t)\tilde{X}_{t}%
-H_{y}(t)\tilde{Y}_{t}-H_{v}(t)(u_{t}-u_{t}^{\ast})\\
&  \geq0,
\end{align*}
where $H(t)=H(t,X_{t}^{\ast},Y_{t}^{\ast},u_{t}^{\ast},p_{t},q_{t})$, similar
definition for $H_{x}(t)$, $H_{y}(t)$ and $H_{v}(t)$. Thus we obtain
$\tilde{Y}_{0}\geq0$ by (\ref{ne-30}), (\ref{ne-31}) and $\Lambda \geq0$, which
implies that $u^{\ast}$ is an optimal control.
\end{proof}

\section{A linear quadratic control problem}

In the following, we suppose that $d=1$ and $\tilde{G}(\cdot)$ is continuously
differentiable. Now we give a characterization of $P^{\ast}\in \mathcal{P}%
^{\ast}$ which is defined in (\ref{ne-20}).

\begin{proposition}
\label{pr-11}Let $\xi=\xi^{\prime}+\int_{0}^{T}\eta_{t}d\langle B\rangle
_{t}-\int_{0}^{T}\tilde{G}(2\eta_{t})dt$, where $\xi^{\prime}\in L_{G}%
^{1}(\Omega_{T})$ such that $\mathbb{\hat{E}}[\xi^{\prime}]=-\mathbb{\hat{E}%
}[-\xi^{\prime}]$, $\eta \in M_{G}^{1}(0,T)$. If $P\in \mathcal{P}$ satisfies
$E_{P}[\xi]-\alpha_{0}^{T}(P)=\mathbb{\tilde{E}}[\xi]$, then $\gamma
_{t}=2\tilde{G}^{\prime}(2\eta_{t})$, a.e., $P$-a.s.
\end{proposition}

\begin{proof}
By Proposition 1.3.8 in \cite{P2019} and Lemma 5.7 in \cite{HJL}, we have%
\[
\mathbb{\tilde{E}}[\xi]=\mathbb{\hat{E}}[\xi^{\prime}]\text{ and }E_{P}%
[\xi]=\mathbb{\hat{E}}[\xi^{\prime}]+E_{P}\left[  \int_{0}^{T}\eta_{t}d\langle
B\rangle_{t}-\int_{0}^{T}\tilde{G}(2\eta_{t})dt\right]  .
\]
Thus, for any $\tilde{\eta}\in M_{G}^{1}(0,T)$, we have
\[
\alpha_{0}^{T}(P)=E_{P}\left[  \int_{0}^{T}\eta_{t}d\langle B\rangle_{t}%
-\int_{0}^{T}\tilde{G}(2\eta_{t})dt\right]  \geq E_{P}\left[  \int_{0}%
^{T}\tilde{\eta}_{t}d\langle B\rangle_{t}-\int_{0}^{T}\tilde{G}(2\tilde{\eta
}_{t})dt\right]  .
\]
Since $\tilde{\eta}$ is arbitrary, we get%
\[
\eta_{t}\gamma_{t}-\tilde{G}(2\eta_{t})\geq a\gamma_{t}-\tilde{G}(2a),\text{
}\forall a\in \mathbb{R},\text{ a.e., }P\text{-a.s.,}%
\]
which implies $\gamma_{t}=2\tilde{G}^{\prime}(2\eta_{t})$, a.e., $P$-a.s.
\end{proof}

We consider the following linear quadratic control (LQ) problem:%
\[
\left \{
\begin{array}
[c]{rl}%
dX_{t}^{u}= & [A(t)X_{t}^{u}+B(t)u_{t}+b(t)]dt+[C(t)X_{t}^{u}+D(t)u_{t}%
+\sigma(t)]dB_{t},\\
X_{0}^{u}= & x_{0}\in \mathbb{R}^{n},
\end{array}
\right.
\]%
\[
Y_{t}^{u}=\mathbb{\tilde{E}}_{t}\left[  \frac{1}{2}\langle LX_{T}^{u}%
,X_{T}^{u}\rangle+\int_{t}^{T}\{E(s)Y_{s}^{u}+\frac{1}{2}[\langle
Q(s)X_{s}^{u},X_{s}^{u}\rangle+2\langle S(s)X_{s}^{u},u_{s}\rangle+\langle
R(s)u_{s},u_{s}\rangle]\}ds\right]  ,
\]
where $\mathcal{U}[0,T]=M_{G}^{2}(0,T;\mathbb{R}^{m})$ and $A$, $B$, $b$, $C$,
$D$, $\sigma$, $E$, $Q$, $S$, $R$ ($t$ is suppressed for simplicity) are
deterministic functions satisfying the following conditions:

\begin{description}
\item[(A6)] $A$, $C\in L^{\infty}(0,T;\mathbb{R}^{n\times n})$, $B\in$
$L^{\infty}(0,T;\mathbb{R}^{n\times m}),$ $D\in C(0,T;\mathbb{R}^{n\times m}%
)$, $E\in L^{\infty}(0,T;\mathbb{R})$, $Q\in L^{\infty}(0,T;\mathbb{S}_{n})$,
$S\in L^{\infty}(0,T;\mathbb{R}^{m\times n})$, $R\in C(0,T;\mathbb{S}_{m})$,
$L\in \mathbb{S}_{n}$, $b$, $\sigma \in L^{2}(0,T;\mathbb{R}^{n}).$

\item[(A7)] $L\geq0$, $Q-S^{T}R^{-1}S\geq0$, $R\gg0$, i.e., there exists a
$\delta>0$ such that $R\geq \delta I_{m\times m}$.
\end{description}

The cost functional is $J(u):=Y_{0}^{u}$, and the LQ problem is to minimize
$J(u)$ over $\mathcal{U}[0,T]$. The following example shows that the Riccati
equation may not exist.

\begin{example}
\label{ex-12}Suppose that $n=m=d=1$, $\tilde{G}(\cdot)$ is convex and not
positive homogeneity. Consider the following LQ problem:%
\[
\left \{
\begin{array}
[c]{rl}%
dX_{t}^{u}= & X_{t}^{u}dB_{t},\text{ }X_{0}^{u}=x_{0}\in \mathbb{R},\\
Y_{t}^{u}= & \mathbb{\tilde{E}}_{t}\left[  \frac{1}{2}|X_{T}^{u}|^{2}+\int
_{t}^{T}|u_{s}|^{2}ds\right]  .
\end{array}
\right.
\]
It is clear that the optimal control $u^{\ast}=0$. Suppose that $Y_{t}^{\ast
}=\frac{1}{2}P(t)|X_{t}^{\ast}|^{2}$, where $P(\cdot)$ satisfies a Riccati
equation. Applying It\^{o}'s formula to $\frac{1}{2}P(t)|X_{t}^{\ast}|^{2}$,
we can get%
\begin{align*}
&  Y_{t}^{\ast}+\int_{t}^{T}P(s)|X_{s}^{\ast}|^{2}dB_{s}+\frac{1}{2}\int
_{t}^{T}P(s)|X_{s}^{\ast}|^{2}d\langle B\rangle_{s}-\int_{t}^{T}\tilde
{G}(P(s)|X_{s}^{\ast}|^{2})ds\\
&  =\frac{1}{2}|X_{T}^{\ast}|^{2}-\int_{t}^{T}[\frac{1}{2}\dot{P}%
(s)|X_{s}^{\ast}|^{2}+\tilde{G}(P(s)|X_{s}^{\ast}|^{2})]ds.
\end{align*}
By Lemma 5.7 in \cite{HJL}, we have $Y_{t}^{\ast}=\mathbb{\tilde{E}}%
_{t}\left[  \frac{1}{2}|X_{T}^{\ast}|^{2}-\int_{t}^{T}[\frac{1}{2}\dot
{P}(s)|X_{s}^{\ast}|^{2}+\tilde{G}(P(s)|X_{s}^{\ast}|^{2})]ds\right]  $. Thus
we deduce%
\begin{equation}
\frac{1}{2}\dot{P}(s)|X_{s}^{\ast}|^{2}+\tilde{G}(P(s)|X_{s}^{\ast}|^{2})=0.
\label{ne-32}%
\end{equation}
Since $\tilde{G}(\cdot)$ is not positive homogeneity, the relation
(\ref{ne-32}) cannot hold. Thus the Riccati equation does not exist in this case.
\end{example}

In order to obtain the Riccati equation, we need an additional assumption
which will be given in the following analysis. The Hamiltonian $H$ for the
above LQ problem is%
\[%
\begin{array}
[c]{rl}%
H(t,x,y,v,p,q)= & p^{T}(Ax+Bv+b)+q^{T}(Cx+Dv+\sigma)\gamma_{t}+Ey\\
& +\frac{1}{2}[\langle Qx,x\rangle+2\langle Sx,v\rangle+\langle Rv,v\rangle].
\end{array}
\]
Let $u^{\ast}$ be an optimal control. By Theorem \ref{th-9}, there exists a
$P^{\ast}\in \mathcal{P}^{\ast}$ satisfying%
\begin{equation}
B^{T}p_{t}+\gamma_{t}D^{T}q_{t}+SX_{t}^{\ast}+Ru_{t}^{\ast}=0\text{, a.e.,
}P^{\ast}\text{-a.s.,} \label{ne-33}%
\end{equation}
where the adjoint equation under $P^{\ast}$ is%
\[
\left \{
\begin{array}
[c]{rl}%
dp_{t}= & -[A^{T}p_{t}+Ep_{t}+\gamma_{t}C^{T}q_{t}+QX_{t}^{\ast}+S^{T}%
u_{t}^{\ast}]dt+q_{t}dB_{t}+dN_{t},\\
p_{T}= & LX_{T}^{\ast}.
\end{array}
\right.
\]
Assume $(\gamma_{t})_{t\leq T}$ is a deterministic function under $P^{\ast}$
and%
\begin{equation}
p_{t}=P(t)X_{t}^{\ast}+\varphi(t), \label{ne-34}%
\end{equation}
where $P(\cdot)\in C^{1}(0,T;\mathbb{S}_{n})$ and $\varphi(\cdot)\in
C^{1}(0,T;\mathbb{R}^{n})$. Applying It\^{o}'s formula to $P(t)X_{t}^{\ast
}+\varphi(t)$, we deduce%
\begin{equation}
\left \{
\begin{array}
[c]{l}%
q_{t}=PCX_{t}^{\ast}+PDu_{t}^{\ast}+P\sigma,\\
\dot{P}X_{t}^{\ast}+PAX_{t}^{\ast}+PBu_{t}^{\ast}+Pb+\dot{\varphi}+A^{T}%
p_{t}+Ep_{t}+\gamma_{t}C^{T}q_{t}+QX_{t}^{\ast}+S^{T}u_{t}^{\ast}=0.
\end{array}
\right.  \label{ne-35}%
\end{equation}
It follows from (\ref{ne-33})-(\ref{ne-35}) that the optimal control%
\begin{equation}
u_{t}^{\ast}=-(R+\gamma_{t}D^{T}PD)^{-1}[(B^{T}P+S+\gamma_{t}D^{T}%
PC)X_{t}^{\ast}+B^{T}\varphi+\gamma_{t}D^{T}P\sigma] \label{ne-36}%
\end{equation}
and $P(\cdot)$ satisfies the following Riccati equation%
\begin{equation}
\left \{
\begin{array}
[c]{l}%
\dot{P}+PA+A^{T}P+EP+\gamma_{t}C^{T}PC+Q\\
-(B^{T}P+S+\gamma_{t}D^{T}PC)^{T}(R+\gamma_{t}D^{T}PD)^{-1}(B^{T}%
P+S+\gamma_{t}D^{T}PC)=0,\text{ a.e. }t\in \lbrack0,T],\\
P(T)=L,
\end{array}
\right.  \label{ne-37}%
\end{equation}%
\begin{equation}
\left \{
\begin{array}
[c]{l}%
\dot{\varphi}+[A^{T}+EI_{n\times n}-(PB+S^{T}+\gamma_{t}C^{T}PD)(R+\gamma
_{t}D^{T}PD)^{-1}B^{T}]\varphi \\
+\gamma_{t}[C^{T}-(PB+S^{T}+\gamma_{t}C^{T}PD)(R+\gamma_{t}D^{T}PD)^{-1}%
D^{T}]P\sigma+Pb=0,\text{ a.e. }t\in \lbrack0,T],\\
\varphi(T)=0.
\end{array}
\right.  \label{ne-38}%
\end{equation}
By the HJB equation (5.1) in \cite{HJL}, Example \ref{ex-12} and
(\ref{ne-36}), we need the following assumption:%
\begin{equation}
C-D(R+\gamma_{t}D^{T}PD)^{-1}(B^{T}P+S+\gamma_{t}D^{T}PC)=0. \label{ne-39}%
\end{equation}

\begin{theorem}
\label{th-13}Let assumptions (A6)-(A7) hold. If there exists a measurable
deterministic function $(\gamma_{t})_{t\leq T}$ satisfying (\ref{ne-39}) and%
\begin{equation}
\gamma_{t}=2\tilde{G}^{\prime}(\langle P\lambda_{t},\lambda_{t}\rangle),
\label{ne-40}%
\end{equation}
where $\lambda_{t}=\sigma-D(R+\gamma_{t}D^{T}PD)^{-1}(B^{T}\varphi+\gamma
_{t}D^{T}P\sigma)$, $P(\cdot)$ and $\varphi(\cdot)$ are solutions for
(\ref{ne-37}) and (\ref{ne-38}), then $u^{\ast}$ defined in (\ref{ne-36}) is
an optimal control for the LQ problem.
\end{theorem}

\begin{proof}
Applying It\^{o}'s formula to $\frac{1}{2}\langle P(t)X_{t}^{\ast},X_{t}%
^{\ast}\rangle+\langle \varphi(t),X_{t}^{\ast}\rangle+l(t)$, where $X^{\ast}$
is the solution of (\ref{ne-2}) corresponding to $u^{\ast}$ and $l(\cdot)$
satisfies the following ODE%
\begin{equation}
\left \{
\begin{array}
[c]{l}%
\dot{l}+El+\langle \varphi,b\rangle-(B^{T}\varphi+\gamma_{t}D^{T}P\sigma
)^{T}(R+\gamma_{t}D^{T}PD)^{-1}B^{T}\varphi+\tilde{G}(\langle P\lambda
_{t},\lambda_{t}\rangle)\\
+\frac{1}{2}(B^{T}\varphi+\gamma_{t}D^{T}P\sigma)^{T}(R+\gamma_{t}%
D^{T}PD)^{-1}R(R+\gamma_{t}D^{T}PD)^{-1}(B^{T}\varphi+\gamma_{t}D^{T}%
P\sigma)=0,\\
l(T)=0,
\end{array}
\right.  \label{ne-41}%
\end{equation}
we can deduce that the solution of (\ref{ne-3}) corresponding to $u^{\ast}$ is
$Y_{t}^{\ast}=\frac{1}{2}\langle P(t)X_{t}^{\ast},X_{t}^{\ast}\rangle
+\langle \varphi(t),X_{t}^{\ast}\rangle+l(t)$ and%
\[
\xi^{\ast}=Y_{0}^{\ast}+\int_{0}^{T}\langle P(t)X_{t}^{\ast},\lambda
_{t}\rangle dB_{t}+\int_{0}^{T}\frac{1}{2}\langle P\lambda_{t},\lambda
_{t}\rangle d\langle B\rangle_{t}-\int_{0}^{T}\tilde{G}(\langle P\lambda
_{t},\lambda_{t}\rangle)dt.
\]
By Proposition \ref{pr-11}, we obtain $d\langle B\rangle_{t}=\gamma_{t}dt$
under $P^{\ast}$. It is easy to verify that
\[
H_{v}(t,X_{t}^{\ast},Y_{t}^{\ast},u_{t}^{\ast},p_{t},q_{t})=0\text{, a.e.,
}P^{\ast}\text{-a.s.}%
\]
Since $H(\cdot)$ is convex with respect to $x$, $y$, $v$ and $\Phi(\cdot)$ is
convex with respect to $x$, we obtain that $u^{\ast}$ is an optimal control by
Theorem \ref{th-10}.
\end{proof}

In the following examples, we can find $\gamma_{t}$.

\begin{example}
Suppose that $\sigma(\cdot)=B(\cdot)=0$. Then $\lambda_{t}=0$ and $\gamma
_{t}=2\tilde{G}^{\prime}(0)$ for $t\in \lbrack0,T]$.
\end{example}

\begin{example}
Suppose that $n=m=1$ and $G(a)=\frac{1}{2}(\bar{\sigma}^{2}a^{+}%
-\underline{\sigma}^{2}a^{-})$ for $a\in \mathbb{R}$. Since $\tilde{G}%
(a_{1})-\tilde{G}(a_{2})\leq G(a_{1}-a_{2})$ for each $a_{1}$, $a_{2}%
\in \mathbb{R}$, we have $2\tilde{G}^{\prime}(\cdot)\in \lbrack \underline
{\sigma}^{2},\bar{\sigma}^{2}]$. Under condition (\ref{ne-39}) and $n=m=1$, it
is easy to verify that the equations for $P(\cdot)$ and $\varphi(\cdot)$ are
independent of $\gamma_{t}$. Thus we can obtain $P(\cdot)$ and $\varphi
(\cdot)$ first. For each fixed $t\in \lbrack0,T]$, $\rho(\gamma_{t}%
):=2\tilde{G}^{\prime}(\langle P\lambda_{t},\lambda_{t}\rangle)-\gamma_{t}$ is
a continuous function with $\rho(\bar{\sigma}^{2})\leq0$ and $\rho
(\underline{\sigma}^{2})\geq0$. Then we get that $\{ \gamma_{t}\in
\lbrack \underline{\sigma}^{2},\bar{\sigma}^{2}]:\rho(\gamma_{t})=0\}$ is a
nonempty closed set. By measurable selection theorem, there exists a
measurable deterministic function $(\gamma_{t})_{t\leq T}$ satisfying
(\ref{ne-40}).
\end{example}

\end{document}